\def\bu{\bullet}
\def\marker{\>\hbox{${\vcenter{\vbox{
    \hrule height 0.4pt\hbox{\vrule width 0.4pt height 6pt
    \kern6pt\vrule width 0.4pt}\hrule height 0.4pt}}}$}\>}
\def\gpic#1{#1
     \smallskip\par\noindent{\centerline{\box\graph}} \medskip}

\documentclass[12pt]{article}
\usepackage{amsmath,amssymb,amsthm}
\usepackage{latexsym}
\usepackage{graphics}
\usepackage{graphicx}
\usepackage{subfigure}
\usepackage{fullpage}
\newtheorem{theorem}{Theorem}[section]

\newtheorem{conjecture}[theorem]{Conjecture}

\newtheorem{lemma}[theorem]{Lemma}

\def\VEC#1#2#3{#1_{#2},\ldots,#1_{#3}}

\def\CH#1#2{\binom{#1}{#2}}
\def\FR#1#2{\frac{#1}{#2}}
\def\FL#1{\left\lfloor{#1}\right\rfloor}

\def\NN{{\mathbb N}}  
  
\def\la{\langle}
\def\ra{\rangle}

\def\cF{{\mathcal{F}}}

\def\cL{{\mathcal{L}}}
\def\cM{{\mathcal{M}}}

\def\cR{{\mathcal{R}}}
\def\esub{\subseteq}
\def\eps{\varepsilon}

\begin{document}

\title{Lichiardopol's conjecture on disjoint cycles in tournaments}

\author{
Fuhong Ma\thanks{School of Mathematics, Shandong University, Jinan 250100,
China: \texttt{mafuhongsdnu@163.com}.}\,,
Douglas B. West\thanks{Departments of Mathematics, Zhejiang Normal University,
Jinhua 321004 China, and University of Illinois at Urbana--Champaign, Urbana IL
61801 USA:
\texttt{dwest@math.uiuc.edu}.  Research supported by National Natural Science
Foundation of China grants NNSFC 11871439 and 11971439.}\,,
Jin Yan\thanks{School of Mathematics, Shandong University, Jinan 250100, China:
\texttt{yanj@sdu.edu.cn}, Corresponding author.  Research supported by National
Natural Science Foundation of China grants NNSFC 11671232, 11271230).}
}

\date{\today}
\maketitle

\baselineskip 16pt

\begin{abstract}
In 2010, N. Lichiardopol conjectured for $q \geq 3$ and $k \geq 1$ that any
tournament with minimum out-degree at least $(q-1)k-1$ contains $k$ disjoint
cycles of length $q$.  We prove this conjecture for $q \geq 5$.  Since it is
already known to hold for $q\le4$, this completes the proof of the conjecture.
\vspace{3mm}


\end{abstract}

\section{Introduction}

We consider cycles in digraphs (directed graphs); a cycle is a strongly
connected digraph in which every vertex has indegree $1$ and outdegree $1$.
The {\it length} of a cycle is the number of edges;
a \emph{$q$-cycle} is a cycle of length $q$.  By ``$k$ disjoint cycles''
we always mean $k$ pairwise vertex disjoint cycles.  A \emph{tournament} is a
digraph obtained from  a complete graph by assigning a direction to each edge.

A famous conjecture of Bermond and Thomassen \cite{J.C.Bermond} for arbitrary
digraphs asserts that large minimum outdegree guarantees many disjoint cycles.

\begin{conjecture}[\rm\cite{J.C.Bermond}]\label{ConjC}
If a digraph $D$ has minimum outdegree at least $2k-1$, then $D$ contains
$k$ disjoint cycles.
\end{conjecture}

This conjecture is trivial for $k=1$.  It was proved for $k=2$ by
Thomassen~\cite{C.Thomassen} and for $k=3$ by Lichiardopol, P\'or, and
Sereni~\cite{N.Lichiardopol1}.  Bang-Jensen, Bessy, and
Thomass\'e~\cite{Jorgen} proved it for the special case of tournaments.
Instead of considering special values of $k$ or special classes of digraphs,
one can also seek to reduce the minimum outdegree needed to guarantee $k$
disjoint cycles.  Alon~\cite{Alon} showed that $64k$ suffices, and later
Buci\'c~\cite{Bucic} reduced this to $18k$.

For general digraphs, Conjecture \ref{ConjC} remains open.  In the special case
of tournaments, stronger results are possible.
Lichiardopol~\cite{N.Lichiardopol2} conjectured that with large minimum
outdegree, one can control not only the number of disjoint cycles but also
their length.

\begin{conjecture}[\rm\cite{N.Lichiardopol2}]\label{ConjB}
If $q \geq 3$ and $k \geq 1$, then every tournament with minimum outdegree
at least $(q-1)k-1$ contains $k$ disjoint $q$-cycles.
\end{conjecture}

In a tournament, disjoint cycles of any length lead to disjoint $3$-cycles
by using chords.  Hence the result of~\cite{Jorgen} yields the special case
of Conjecture~\ref{ConjB} for $q=3$.

\begin{theorem}[\rm\cite{Jorgen}]\label{ThA}
Every tournament $T$ with minimum outdegree at least $2k-1$ has $k$ disjoint
$3$-cycles.
\end{theorem}

Bessy, Lichiardopol, and Sereni~\cite{BLS} had earlier proved that every
tournament with minimum indegree and outdegree both at least $2k-1$ has $k$
disjoint cycles.  In support of Conjecture~\ref{ConjB},
Lichiardopol~\cite{N.Lichiardopol2} proved the two weaker theorems below, the
first of which improves the result of~\cite{BLS}.

\begin{theorem}[\rm\cite{N.Lichiardopol2}]\label{Tha}
If $q \geq 3$ and $k \geq 1$, then every tournament with minimum outdegree and
indegree both at least $(q-1)k-1$ contains $k$ disjoint $q$-cycles.
\end{theorem}

Ma and Yan~\cite{yan} improved Theorem~\ref{Tha} by guaranteeing more than
$k$ disjoint cycles under the same conditions, so the conclusion of
Theorem~\ref{Tha} is not sharp (\cite{Ma} addressed the special case of
$4$-cycles in regular tournaments).

\begin{theorem}[\rm\cite{N.Lichiardopol2}]\label{CORO}
If $q \geq 3$ and $k \geq 1$, then every tournament with minimum outdegree at
least $(q-1)k-1$ contains $\lceil k-1-\frac{k-2}{q} \rceil$ disjoint $q$-cycles.
\end{theorem}

The case $q=4$ of Conjecture~\ref{ConjB} was proved in the masters thesis
of S. Zhu~\cite{Zhu}.  Hence our result in this paper completes the proof of
Conjecture~\ref{ConjB}.  Using some of our lemmas and similar methods,
Wang, Ma, and Yan~\cite{WMY} gave an independent proof of the cases $q\le9$.
Our result for $q\ge5$ is self-contained and does not use any of their
arguments.

\begin{theorem}\label{ThB}
For $q \geq 5$ and $k \geq 1$, every tournament with minimum outdegree at
least $(q-1)k-1$ contains $k$ disjoint $q$-cycles.
\end{theorem}

For Conjecture~\ref{ConjB} (and Theorem~\ref{ThB}), the degree hypothesis is
not known to be sharp.  A trivial lower bound on the minimum outdegree needed
to guarantee the conclusion is $qk/2$, since a tournament on $qk-1$ vertices in
which every vertex has outdegree $qk/2-1$ does not have enough vertices to have
$k$ disjoint $q$-cycles.  Given that the needed inequalities are easier to
satisfy when $q$ is large, we ask whether there is a positive constant $\eps$
such that minimum outdegree $(1-\eps)qk$ suffices when $q$ is sufficiently
large.

Motivated by this problem on tournaments, one may wonder whether large
outdegree can guarantee disjoint cycles of the same length in general digraphs,
even without constraining which length it is.  Our result guarantees this
for tournaments.  Thomassen~\cite{C.Thomassen} conjectured such a relationship
for general digraphs, but Alon~\cite{Alon} showed that it cannot hold.

\begin{theorem}[\rm\cite{Alon}]\label{PropA}
For all $r\in\NN$, some digraph with minimum outdegree $r$ has no two
edge-disjoint cycles of the same length (and hence also no disjoint
cycles of the same length).
\end{theorem}

\medskip
Other papers on disjoint cycles in digraphs include \cite{BAI} and \cite{CHEN}.  
Disjoint cycles have also been studied in
undirected graphs, where the results are more plentiful.  The
Bermond--Thomassen Conjecture is in fact the directed
analogue of the Corr\'adi--Hajnal Theorem~\cite{CorHaj}.

\begin{theorem}[\rm\cite{CorHaj}]
For $n,k\in\NN$ with $n\ge 3k$, every $n$-vertex undirected graph with
minimum degree at least $2k$ contains $k$ disjoint cycles.
\end{theorem}

There are many extensions and variations on this result.  Those most similar
to our work consider the lengths of disjoint cycles guaranteed by a threshold
on the minimum degree.  Let $\delta(G)$ denote the minimum degree of a graph
$G$; also, the {\it order} of a graph or digraph is the number of vertices.

Thomassen~\cite{Thom2} proved for $k\ge2$ that every graph $G$ with
$\delta(G)\ge 3k+1$ and order at least some constant $c_k$ contains $k$
disjoint cycles of the same length.  Thomassen conjectured that minimum degree
$2k$ suffices, which had earlier been conjectured for $k=2$ by H\"aggkvist.
Egawa~\cite{Egawa} proved Thomassen's conjecture for $k\ge3$ with a threshold
of $|V(G)|\ge 17k+o(k)$.  Verstra\"ete~\cite{Verstraete} later proved
Thomassen's conjecture in full.  Verstra\"ete also conjectured that order at
least $4k$ is enough to guarantee $k$ disjoint cycles of the same length
when $\delta(G)\ge2k$.

Chiba, Fujita, Kawarabayashi, and Sakuma~\cite{Chiba} guaranteed for $k\in\NN$
a constant $c_k$ such that every graph with order at least $c_k$ and minimum
degree at least $2k$ contains $k$ disjoint even cycles, with special exceptions.
Other degree conditions for disjoint cycles in undirected graphs can be found
in~\cite{BHLLL} and in the survey~\cite{ChiYam}.

\section{Structure of the Proof}

To prove Theorem~\ref{ThB}, we prove a theorem that was mainly inspired by
the proof of the Bermond--Thomassen Conjecture for tournaments by
Bang-Jensen, Bessy, and Thomass\'e \cite{Jorgen}.
\looseness -1

\begin{theorem}\label{ThC}
Fix $k,q\in\NN$ with $q\ge5$ and $k\ge2$, and let $T$ be a tournament with
minimum outdegree at least $(q-1)k-1$.  For any family ${\cF}$ of $k-1$
disjoint $q$-cycles in $T$, there is a family of $k$ disjoint $q$-cycles in $T$
using at most $3q-6$ vertices outside the cycles in ${\cF}$.
\end{theorem}

We show first that this suffices.  For a path with vertices $\VEC v1q$ in
order, we use the notation $\la \VEC v1q\ra$ (each edge in $\la \VEC v1q\ra$
is oriented from $v_i$ to $v_{i+1}$); we also call this a {\it $v_1,v_q$-path}.
For a cycle with vertices $\VEC v1q$ in order, we use the notation $[\VEC v1q]$.
The outdegree of a vertex $v$ is $d^+(v)$, and the minimum outdegree in a
digraph $D$ is $\delta^+(D)$.

It is well-known that every tournament contains a spanning path (R\'edei's
Theorem~\cite{Redei}), and that a tournament is strongly connected if and only
if it contains a spanning cycle (Camion's Theorem~\cite{Camion}), where a
digraph is {\it strongly connected} or {\it strong} if it contains a $u,v$-path
for any two vertices $u$ and $v$.  Moreover, a strong touranment
(or the subtournament induced by any cycle) is {\it pancyclic}, meaning that
it contains cycles of all lengths from $3$ through through the number of
vertices.  When invoking this, we say ``by pancyclicity''.
Finally, Moon~\cite{Moon} showed that every strong tournament with
at least three vertices is {\it vertex pancyclic}, meaning that through any
vertex there are cycles of all lengths.

\begin{lemma}
Theorem~\ref{ThC} implies Theorem \ref{ThB}.
\end{lemma}
\begin{proof}
Assuming that Theorem~\ref{ThC} holds, we prove Theorem~\ref{ThB} by induction
on $k$.  When $k = 1$, we are given a tournament $T$ with $\delta^+(T)\ge q-2$.
Since every tournament has a spanning path, we may let $\la \VEC v1n\ra$ be a
spanning path of $T$.  Since $d^+(v_n)\ge q-2$ and $v_{n-1}$ is not an
outneighbor of $v_n$, the edge to the earliest outneighbor of $v_n$ along $P$
completes a cycle of length at least $q$ in $T$.  By pancyclicity, $T$
contains a $q$-cycle.

For the induction step, suppose $k>1$.  We have
$\delta^+(T) \geq (q-1)k-1>(q-1)(k-1)-1$.  By the induction hypothesis,
$T$ contains $k-1$ disjoint $q$-cycles.  From this family ${\cF}$
of $k-1$ disjoint $q$-cycles, Theorem \ref{ThC} produces a family of $k$
disjoint $q$-cycles.
\end{proof}

We will prove Theorem~\ref{ThC} by considering two cases.  In
Theorem~\ref{THMA}, we will prove that the conclusion holds when $q\ge5$ and
$k\le q$.  In Theorem~\ref{THMB}, we will prove by induction on $k$ that the
conclusion holds when $q\ge5$ and $k>q$, using Theorem~\ref{THMA} as a basis.
The restriction to using $3q-6$ vertices outside ${\cF}$ will be helpful in
proving Theorem~\ref{THMB}.

\section{Cycles and Paths in Tournaments}

Here we prove some structural lemmas about tournaments that will be useful
in the proof of Theorem \ref{THMA}.  We use $T[S]$ to denote the subtournament
of $T$ induced by the vertex set $S$.  We also let $d^+(X,Y)$ denote the
number of edges with tail in $X$ and head in $Y$, where $X$ and $Y$ may be
a set of vertices or a single vertex.  When $uv$ is an edge, we say that
$v$ is a {\it successor} of $u$ and $u$ is a {\it predecessor} of $v$.

\begin{lemma}\label{lem1}
If $C$ is a cycle of length $m$ in a tournament $T$, where $m\ge4$, then
$T[V(C)]$ contains a cycle $C'$ of length $m-1$ such that the omitted vertex
$u$ of $C$ has at least two predecessors in $V(C')$.
\end{lemma}

\vspace{-.7pc}
\begin{proof}
Since every strong tournament is pancyclic,
$T[V(C)]$ contains a cycle $C^*$ of length $m-1$.  Let $u'$ be the vertex of
$C$ not in $C^*$, and let $v$ be the predecessor of $u'$ on $C$.  If $u'$ has
another predecessor on $C^*$, then let $u=u'$ and $C'=C^*$.  Otherwise, let
$\la v,w,x\ra$ be the portion of $C^*$ leaving $v$ (this exists since $m\ge4$).
Since $u'$ has only one predecessor in $V(C)$, both $u'w$ and $u'x$ are edges.
Now form $C'$ by replacing $w$ with $u'$ in $C^*$, and let $u=w$.  The omitted
vertex $u$ now has predecessors $v$ and $u'$ on $C'$.
\end{proof}

We believe in general that for $m\ge2r$, a strong tournament with $m$ vertices
contains a cycle of length $m-1$ such that the omitted vertex has at least $r$
predecessors in $C'$.  We proved this for $r=3$, but the proof is considerably
longer than for $r=2$, and we only need the result for $r=2$.  The threshold on
$m$ is sharp; a tournament on $2r-1$ vertices in which every vertex has $r-1$
predecessors and $r-1$ successors has no vertex with $r$ predecessors.

\begin{lemma}\label{lem2}
If $C$ is a cycle of length $m$ in a tournament $T$, where $m\ge5$, then
$T[V(C)]$ contains a cycle $C''$ of length $m-2$ omitting vertices $x$ and $y$
with $xy\in E(T)$ such that $y$ has at least one predecessor in $V(C'')$.
\end{lemma}

\vspace{-.7pc}
\begin{proof}
We use Lemma~\ref{lem1} twice.  First, it gives us a cycle
$C'$ in $T[V(C)]$, with length $m-1$, such that the remaining vertex $u$ has at
least two predecessors in $V(C')$.  We then apply Lemma~\ref{lem1} using
this $C'$ as $C$; it gives us a cycle $C''$ in $T[V(C')]$, with length $m-2$,
such that the remaining vertex $v$ has at least two predecessors in $V(C'')$.
From $C$, the vertices $u$ and $v$ have been omitted.  Each has at least
one predecessor in $V(C'')$.  Hence we let $x$ be the tail and $y$ be the
head of the edge joining $u$ and $v$.
%
\end{proof}


We say that a vertex set $X$ {\it dominates} a vertex set $Y$ in a tournament
if every edge joining $X$ and $Y$ is oriented from $X$ to $Y$.

\begin{lemma}\label{lem3}
If $C$ is a cycle of length $m$ in a tournament $T$, where $m\ge4$,
and $S$ is a $3$-set in $V(C)$, then at least one vertex of $S$ has at
least two predecessors in $V(C)$.
\end{lemma}

\vspace{-.7pc}
\begin{proof}
If the claim fails, then $T[S]$ is a $3$-cycle.  Now failure of the claim
requires $S$ to dominate $V(C)-S$, contradicting that $C$ is a cycle.
\end{proof}

\begin{lemma}\label{lem4}
Let $T$ be a tournament with minimum outdegree at least $(q-1)k-1$, where
$2\le k\le q$ and $q\ge5$.  Let ${\cF}$ be a family of $k-1$ disjoint
$q$-cycles in $T$, with vertex sets $\VEC V1{k-1}$ inducing cycles
$\VEC C1{k-1}$, and let
$P$ be a path through the remaining vertices.
If $S_1$ and $S_2$ partition $V(P)$ with $S_2$ dominating $S_1$ in $T$
and $|S_1|\le q$, then

(a) $d^+(V^*,S_1) \le \frac{|S_1|+1}{2q-2}|S_1|$
for some $V^* \in \{\VEC V1{k-1}\}$.

(b) If $|S_1|\ge 2$, then $d^+(S_1,z)\ge 2$ for any $z \in V^*$.

(c) If $|S_1|\le q-1$, then $d^+(u,V^*) \geq 3$ for any $u\in S_1$.


(d) If $S_1$ dominates $z \in V^*$, and $z$ has at least $r$ predecessors
in $V^*$, then $d^+(z,S_2) \geq r$.
%
\end{lemma}

\begin{proof}
{\it of (a):}
Since $S_2$ dominates $S_1$,
$$
d^+(S_1,\bigcup V_i) ~\ge~[(q\!-\!1)k-1]|S_1|-\frac{|S_1|(|S_1|\!-\!1)}{2}
~=~ \left[(q\!-\!1)(k\!-\!1)+\frac{2q\!-\!|S_1|\!-\!3}{2}\right]|S_1|.
$$
By the pigeonhole principle, there exists $C_i \in {\cF}$ such that
$d^+(S_1,V_i)\ge (q-1+\frac{2q-|S_1|-3}{2(k-1)})|S_1|$.
Since $|S_1|\le q$ and $q\ge 3$, we have $2q-|S_1|-3\ge 0$.
We then use $k\le q$ to conclude $d^+(S_1,V_i)\ge(q-\frac{|S_1|+1}{2q-2})|S_1|$.
Let $C^*$ be this cycle $C_i$, with vertex set $V^*$.


\bigskip
\noindent
{\it Proof of (b):}
If $d^+(S_1,z)\le 1$ for some $z\in V^*$, then
$d^+(V^*,S_1) \ge d^+(z,S_1)\ge |S_1|-1$.  Using part (a),
we obtain
\begin{equation}\label{E1}
|S_1|-1\le \FR{|S_1|+1}{2q-2}|S_1|.
\end{equation}
With $s=|S_1|$, the inequality can be rewritten as $2q-2-s(2q-3-s)\ge0$.
However, with $2\le s\le q$, the left side of this is negative when $q\ge5$.

\bigskip
\noindent
{\it Proof of (c):}
If $d^+(u,V^*)\le2$ for some $u \in S_1$.
Now $d^+(V^*,S_1)\ge d^+(V^*,u)\ge q-2$.
Using (a) and $|S_1|\le q-1$, we obtain
$q-2\le \FR{q}{2q-2}(q-1)$, which requires $q\le4$.


\bigskip
\noindent
{\it Proof of (d):}
Since $z$ is dominated by $S_1$ and has at least $r$ predecessors in $V^*$,
$$d^+(z,S_2)= d^+(z)-d^+(z,\bigcup_i V_i) \ge (q-1)k-1-[q(k-1)-(r+1)]=q-k+r.$$
Since $k\leq q$, we obtain $d^+(z,S_2)\ge r$.
\end{proof}

\section{The Case of Small $k$}\label{SecA}

In this section we prove Theorem~\ref{ThC} for $k\le q$, stated as
Theorem~\ref{THMA}.  Essentially, we provide an algorithm
to produce the desired family ${\cF}^*$ of $k$ disjoint $q$-cycles by
iteratively increasing the length of a cycle found outside the $k-1$ disjoint
$q$-cycles.  The subtournament $T'$ induced by the vertices not in the given
cycles has a spanning path $P$; let $v$ be its last vertex.  If $v$ lies in a
cycle of length at least $q$ in $T'$, then by pancyclicity there is a $q$-cycle
in $T'$, and we are done.  Hence our approach, given a longest cycle through
$v$ in $T'$ (in the first step the length may be $0$), is to rearrange ${\cF}$
to find a new family ${\cF'}$ of $k-1$ disjoint $q$-cycles so that the vertex
at the end of the resulting remaining path $P'$ lies in a longer cycle.

Some of the claims in this argument are not valid when $q=4$.  Nevertheless,
the same framework applies when $q=4$, with additional more detailed reasoning.

\begin{theorem}\label{THMA}
Given $k,q\in\NN$ with $q \geq 5$ and $k \leq q$, let $T$ be a tournament
with $\delta^+(T)\ge(q-1)k-1$.  For any family ${\cF}$ of $k-1$ disjoint
$q$-cycles in $T$, there is a family $\cF^*$ of $k$ disjoint $q$-cycles in $T$
whose union has at most $3q-6$ vertices outside the cycles in ${\cF}$.
\end{theorem}

\begin{proof}
These hypotheses are the same as those of Lemma~\ref{lem4} once we obtain a
partition $(S_2,S_1)$ of the vertices outside ${\cF}$ such that $S_2$ dominates
$S_1$.  Given such a partition, let $C^*$ with vertex set $V^*$ be the cycle in
$\cF$ guaranteed by Lemma~\ref{lem4}(a).  Let $P$ be a spanning path through
the subtournament $T'$ of vertices not used by $\cF$, with last vertex $v$.

When $T'$ has a cycle through $v$, let $l$ be the maximum length of such a
cycle; otherwise $l=0$.  If $l\ge q$, then pancyclicity of the subtournament
spanned by this cycle provides a $q$-cycle to complete $\cF^*$.
Otherwise, we obtain a new family $\cF'$ where $l$ is larger, generally by
replacing $C^*$ with a new cycle $\widehat C$ and defining a new path $P'$
through the vertices outside $\cF'$.   We will use at most two new vertices in
$\widehat C$ at each step that increases $l$, except that the steps to reach
$l\ge4$ will use at most three new vertices. In addition, when a sufficiently
long cycle appears, it uses at most $q-1$ new vertices, because the
subtournament that was outside the $k-1$ given $q$-cycles entering that step
did not contain a cycle of length at least $q$. Thus in total at most $3q-6$
new vertices are used.

\medskip

{\bf Case 1:}
{\it $l=0$, so $T'$ has no cycle through $v$.}
We seek $\widehat C$ and $P'$ so that there is a cycle outside $\cF'$
through the last vertex of $P'$.  Let $u$ be the predecessor of $v$ on $P$.

\smallskip

{\bf Case 1a.} {\it $T'$ has no cycle through $u$}
(see Figure~\ref{fig1a}).
Let $S_1=\{u,v\}$ and $S_2=V(T')-S_1$.  Since $T'$ has no cycle through $u$
or $v$, $S_2$ dominates $S_1$.  By Lemma~\ref{lem4}(a), there
exists $C^* \in {\cF}$ such that $d^{+}(V^*,S_1)\le \FR3{2q-2}\cdot 2$,
and $\FR3{2q-2}\cdot 2<1$ when $q\ge5$.  Thus $S_1$ dominates $V^*$.

Since $q\ge5$, Lemma~\ref{lem2} implies that $T[V^*]$ contains a cycle $C''$
of length $q-2$ such that $y$ has a predecessor in $V(C'')$, where $xy$ is the
edge in $T[V^*]-V(C'')$.  Choose three vertices in $V(C'')$.  By
Lemma~\ref{lem3}, among them is a vertex $z$ with at least two predecessors in
$V^*$.  Since also $z$ is dominated by $S_1$,
Lemma~\ref{lem4}(d) guarantees $d^+(z,S_2)\ge2$.  Let $w$ be a successor of $z$
in $S_2$, and let $z'$ be the successor of $z$ on $C''$.  Replacing $zz'$ in
$C''$ with $\la z,w,u,z'\ra$ yields a $q$-cycle $\widehat C$ with two vertices
outside ${\cF}$.  Replace $C^*$ with $\widehat C$ to form $\cF'$ from $\cF$.

Since $S_2$ dominates $S_1$, we can form the $P'$ outside $\cF'$ by appending
$v,x,y$ in order to a spanning path of $T[S_2]-w$.  Since $S_1$ dominates
$V^*$, and $y$ has at least $x$ and a vertex of $C''$ as predecessors in $V^*$,
Lemma~\ref{lem4}(d) yields $d^+(y,S_2)\ge2$.  Thus $y$ has a successor in $S_2$
other than $w$, so there is a cycle through $y$ using vertices of $P'$.

\begin{figure}[hbt]
\begin{center}
\gpic{
\expandafter\ifx\csname graph\endcsname\relax \csname newbox\endcsname\graph\fi
\expandafter\ifx\csname graphtemp\endcsname\relax \csname newdimen\endcsname\graphtemp\fi
\setbox\graph=\vtop{\vskip 0pt\hbox{%
    \special{pn 8}%
    \special{pa 131 787}%
    \special{pa 1115 787}%
    \special{fp}%
    \special{sh 1.000}%
    \special{pa 984 754}%
    \special{pa 1115 787}%
    \special{pa 984 820}%
    \special{pa 984 754}%
    \special{fp}%
    \special{pa 1115 787}%
    \special{pa 1770 787}%
    \special{fp}%
    \special{pa 131 131}%
    \special{pa 721 131}%
    \special{fp}%
    \special{sh 1.000}%
    \special{pa 590 98}%
    \special{pa 721 131}%
    \special{pa 590 164}%
    \special{pa 590 98}%
    \special{fp}%
    \special{pa 721 131}%
    \special{pa 1115 131}%
    \special{fp}%
    \special{pa 1115 131}%
    \special{pa 1508 131}%
    \special{fp}%
    \special{sh 1.000}%
    \special{pa 1377 98}%
    \special{pa 1508 131}%
    \special{pa 1377 164}%
    \special{pa 1377 98}%
    \special{fp}%
    \special{pa 1508 131}%
    \special{pa 1770 131}%
    \special{fp}%
    \graphtemp=.5ex\advance\graphtemp by 0.787in
    \rlap{\kern 0.459in\lower\graphtemp\hbox to 0pt{\hss $\bu$\hss}}%
    \graphtemp=.5ex\advance\graphtemp by 0.787in
    \rlap{\kern 1.443in\lower\graphtemp\hbox to 0pt{\hss $\bu$\hss}}%
    \graphtemp=.5ex\advance\graphtemp by 0.787in
    \rlap{\kern 1.770in\lower\graphtemp\hbox to 0pt{\hss $\bu$\hss}}%
    \graphtemp=.5ex\advance\graphtemp by 0.131in
    \rlap{\kern 0.459in\lower\graphtemp\hbox to 0pt{\hss $\bu$\hss}}%
    \graphtemp=.5ex\advance\graphtemp by 0.131in
    \rlap{\kern 0.459in\lower\graphtemp\hbox to 0pt{\hss $\bu$\hss}}%
    \graphtemp=.5ex\advance\graphtemp by 0.131in
    \rlap{\kern 0.787in\lower\graphtemp\hbox to 0pt{\hss $\bu$\hss}}%
    \graphtemp=.5ex\advance\graphtemp by 0.295in
    \rlap{\kern 1.443in\lower\graphtemp\hbox to 0pt{\hss $\bu$\hss}}%
    \graphtemp=.5ex\advance\graphtemp by 0.131in
    \rlap{\kern 1.770in\lower\graphtemp\hbox to 0pt{\hss $\bu$\hss}}%
    \special{pn 28}%
    \special{ar 623 1347 1311 1311 -1.955193 -1.186400}%
    \special{ar 951 -428 1311 1311 1.186400 1.955193}%
    \special{pn 8}%
    \special{pa 1770 787}%
    \special{pa 1574 492}%
    \special{fp}%
    \special{sh 1.000}%
    \special{pa 1619 619}%
    \special{pa 1574 492}%
    \special{pa 1674 583}%
    \special{pa 1619 619}%
    \special{fp}%
    \special{pa 1574 492}%
    \special{pa 1443 295}%
    \special{fp}%
    \special{pa 1443 295}%
    \special{pa 1639 197}%
    \special{fp}%
    \special{sh 1.000}%
    \special{pa 1507 226}%
    \special{pa 1639 197}%
    \special{pa 1537 285}%
    \special{pa 1507 226}%
    \special{fp}%
    \special{pa 1639 197}%
    \special{pa 1770 131}%
    \special{fp}%
    \special{pn 28}%
    \special{pa 459 131}%
    \special{pa 459 787}%
    \special{fp}%
    \special{pn 8}%
    \special{pa 459 393}%
    \special{pa 459 525}%
    \special{fp}%
    \special{sh 1.000}%
    \special{pa 492 393}%
    \special{pa 459 525}%
    \special{pa 426 393}%
    \special{pa 492 393}%
    \special{fp}%
    \special{pn 28}%
    \special{pa 1443 787}%
    \special{pa 787 131}%
    \special{fp}%
    \special{pn 8}%
    \special{pa 1180 525}%
    \special{pa 1049 393}%
    \special{fp}%
    \special{sh 1.000}%
    \special{pa 1119 509}%
    \special{pa 1049 393}%
    \special{pa 1165 463}%
    \special{pa 1119 509}%
    \special{fp}%
    \special{pn 28}%
    \special{pa 131 131}%
    \special{pa 459 131}%
    \special{fp}%
    \special{pa 787 131}%
    \special{pa 1115 131}%
    \special{fp}%
    \special{pn 8}%
    \special{pa 1344 869}%
    \special{pa 1869 869}%
    \special{pa 1869 705}%
    \special{pa 1344 705}%
    \special{pa 1344 869}%
    \special{da 0.066}%
    \graphtemp=.5ex\advance\graphtemp by 0.918in
    \rlap{\kern 0.459in\lower\graphtemp\hbox to 0pt{\hss $w$\hss}}%
    \graphtemp=.5ex\advance\graphtemp by 0.918in
    \rlap{\kern 1.443in\lower\graphtemp\hbox to 0pt{\hss $u$\hss}}%
    \graphtemp=.5ex\advance\graphtemp by 0.918in
    \rlap{\kern 1.770in\lower\graphtemp\hbox to 0pt{\hss $v$\hss}}%
    \graphtemp=.5ex\advance\graphtemp by 0.787in
    \rlap{\kern 2.000in\lower\graphtemp\hbox to 0pt{\hss $S_1$\hss}}%
    \graphtemp=.5ex\advance\graphtemp by 0.787in
    \rlap{\kern 0.000in\lower\graphtemp\hbox to 0pt{\hss $P$\hss}}%
    \graphtemp=.5ex\advance\graphtemp by 0.131in
    \rlap{\kern 0.000in\lower\graphtemp\hbox to 0pt{\hss $C''$\hss}}%
    \graphtemp=.5ex\advance\graphtemp by 0.224in
    \rlap{\kern 0.366in\lower\graphtemp\hbox to 0pt{\hss $z$\hss}}%
    \graphtemp=.5ex\advance\graphtemp by 0.224in
    \rlap{\kern 0.694in\lower\graphtemp\hbox to 0pt{\hss $z'$\hss}}%
    \graphtemp=.5ex\advance\graphtemp by 0.295in
    \rlap{\kern 1.311in\lower\graphtemp\hbox to 0pt{\hss $x$\hss}}%
    \graphtemp=.5ex\advance\graphtemp by 0.131in
    \rlap{\kern 1.902in\lower\graphtemp\hbox to 0pt{\hss $y$\hss}}%
    \hbox{\vrule depth0.918in width0pt height 0pt}%
    \kern 2.000in
  }%
}%
}
\vspace{-1pc}
\caption{Case 1a of Theorem~\ref{THMA}.\label{fig1a}}
\end{center}
\end{figure}
\vspace{-1pc}

\medskip
{\bf Case 1b.} {\it $T'$ has a cycle containing $u$.}
Let $B$ be a longest cycle containing $u$ in $T'$.  Let $S_1=V(B)\cup \{v\}$
and $S_2=V(T')-S_1$.  Any edge from $V(B)$ to $S_2$ yields a larger strong
tournament in $T'$ containing $V(B)$, which contains a larger cycle
containing $u$.  Hence $S_2$ dominates $S_1$.  Also $|V(B)|\le q-1$, since
otherwise we have the $k$th $q$-cycle.  Hence $4\le |S_1|\le q$, so
Lemma~\ref{lem4} applies.  Choose $C^*$ with vertex set $V^*$ as given by
Lemma~\ref{lem4}.

We will use the following ``degree fact''.  If ${\cF'}$ is a family
of $k-1$ disjoint $q$-cycles, and the last vertex $v'$ in a spanning path $P'$
through the set $S$ of remaining vertices has at least two predecessors used by
$\cF'$, then $v'$ has a successor in $S$.  The reason, using $k\le q$, is
$$[(q-1)k-1]-[q(k-1)-2]=q-k+1\ge1.$$
There is then a cycle in $S$ through $v'$, as desired.

The degree fact implies $d^+(V^*,v)\le1$; otherwise already $T'$ has a cycle
through $v$.  First suppose $d^+(V^*,v) = 1$ (see Figure~\ref{fig1b1}).  Let
$z$ be the predecessor of $v$ in $V^*$, and let $y$ be the successor of $z$ on
$C^*$.  Form $\cF'$ by replacing $C^*$ with the cycle $\widehat C$ obtained
from $C^*$ by replacing $y$ with $v$ (the successor of $y$ on $C^*$ is a
successor of $v$).
Since $4\le|S_1|\le q$, by Lemma~\ref{lem4}(b) $y$ has at least two
predecessors in $S_1$, at least one in $B$ (call this vertex $w$).
Since $S_2$ dominates $S_1$, we can form $P'$ by following $P$ through
$S_2$, then $B$ from the successor of $w$ on $B$ to $w$, and finally the
edge $wy$.  Since $v$ and $z$ are predecessors of $y$, the degree fact yields
a cycle outside $\cF'$ through $y$.  The only vertex used by $\cF'$ and not
by $\cF$ is $v$.

\begin{figure}[hbt]
\begin{center}
\gpic{
\expandafter\ifx\csname graph\endcsname\relax \csname newbox\endcsname\graph\fi
\expandafter\ifx\csname graphtemp\endcsname\relax \csname newdimen\endcsname\graphtemp\fi
\setbox\graph=\vtop{\vskip 0pt\hbox{%
    \special{pn 11}%
    \special{ar 1311 840 252 118 0 6.28319}%
    \special{pn 8}%
    \special{pa 134 840}%
    \special{pa 689 840}%
    \special{fp}%
    \special{sh 1.000}%
    \special{pa 555 807}%
    \special{pa 689 840}%
    \special{pa 555 874}%
    \special{pa 555 807}%
    \special{fp}%
    \special{pa 689 840}%
    \special{pa 1059 840}%
    \special{fp}%
    \special{pa 471 168}%
    \special{pa 1076 168}%
    \special{fp}%
    \special{sh 1.000}%
    \special{pa 941 134}%
    \special{pa 1076 168}%
    \special{pa 941 202}%
    \special{pa 941 134}%
    \special{fp}%
    \special{pa 1076 168}%
    \special{pa 1479 168}%
    \special{fp}%
    \special{pn 11}%
    \special{pa 1563 840}%
    \special{pa 1815 840}%
    \special{fp}%
    \graphtemp=.5ex\advance\graphtemp by 0.168in
    \rlap{\kern 1.479in\lower\graphtemp\hbox to 0pt{\hss $\bu$\hss}}%
    \graphtemp=.5ex\advance\graphtemp by 0.757in
    \rlap{\kern 1.133in\lower\graphtemp\hbox to 0pt{\hss $\bu$\hss}}%
    \graphtemp=.5ex\advance\graphtemp by 0.840in
    \rlap{\kern 1.815in\lower\graphtemp\hbox to 0pt{\hss $\bu$\hss}}%
    \graphtemp=.5ex\advance\graphtemp by 0.168in
    \rlap{\kern 0.807in\lower\graphtemp\hbox to 0pt{\hss $\bu$\hss}}%
    \graphtemp=.5ex\advance\graphtemp by 0.168in
    \rlap{\kern 0.807in\lower\graphtemp\hbox to 0pt{\hss $\bu$\hss}}%
    \graphtemp=.5ex\advance\graphtemp by 0.168in
    \rlap{\kern 1.143in\lower\graphtemp\hbox to 0pt{\hss $\bu$\hss}}%
    \graphtemp=.5ex\advance\graphtemp by 0.874in
    \rlap{\kern 1.429in\lower\graphtemp\hbox to 0pt{\hss $B$\hss}}%
    \special{pn 28}%
    \special{ar 1143 1708 1681 1681 -1.982313 -1.159279}%
    \special{pn 11}%
    \special{pa 807 168}%
    \special{pa 1479 168}%
    \special{fp}%
    \special{pn 8}%
    \special{pa 1133 757}%
    \special{pa 1139 404}%
    \special{fp}%
    \special{sh 1.000}%
    \special{pa 1103 538}%
    \special{pa 1139 404}%
    \special{pa 1170 539}%
    \special{pa 1103 538}%
    \special{fp}%
    \special{pa 1139 404}%
    \special{pa 1143 168}%
    \special{fp}%
    \special{pa 1025 1008}%
    \special{pa 1866 1008}%
    \special{pa 1866 672}%
    \special{pa 1025 672}%
    \special{pa 1025 1008}%
    \special{da 0.067}%
    \graphtemp=.5ex\advance\graphtemp by 0.840in
    \rlap{\kern 2.000in\lower\graphtemp\hbox to 0pt{\hss $S_1$\hss}}%
    \special{pn 28}%
    \special{pa 807 168}%
    \special{pa 1815 840}%
    \special{fp}%
    \special{pn 8}%
    \special{pa 1210 437}%
    \special{pa 1412 571}%
    \special{fp}%
    \special{sh 1.000}%
    \special{pa 1319 469}%
    \special{pa 1412 571}%
    \special{pa 1281 525}%
    \special{pa 1319 469}%
    \special{fp}%
    \special{pn 28}%
    \special{pa 1815 840}%
    \special{pa 1479 168}%
    \special{fp}%
    \special{pn 8}%
    \special{pa 1681 571}%
    \special{pa 1613 437}%
    \special{fp}%
    \special{sh 1.000}%
    \special{pa 1644 572}%
    \special{pa 1613 437}%
    \special{pa 1704 542}%
    \special{pa 1644 572}%
    \special{fp}%
    \special{pn 28}%
    \special{pa 471 168}%
    \special{pa 807 168}%
    \special{fp}%
    \special{pa 1479 168}%
    \special{pa 1815 168}%
    \special{fp}%
    \graphtemp=.5ex\advance\graphtemp by 0.263in
    \rlap{\kern 1.238in\lower\graphtemp\hbox to 0pt{\hss $y$\hss}}%
    \graphtemp=.5ex\advance\graphtemp by 0.819in
    \rlap{\kern 1.228in\lower\graphtemp\hbox to 0pt{\hss $w$\hss}}%
    \graphtemp=.5ex\advance\graphtemp by 0.941in
    \rlap{\kern 1.815in\lower\graphtemp\hbox to 0pt{\hss $v$\hss}}%
    \graphtemp=.5ex\advance\graphtemp by 0.840in
    \rlap{\kern 0.000in\lower\graphtemp\hbox to 0pt{\hss $P$\hss}}%
    \graphtemp=.5ex\advance\graphtemp by 0.168in
    \rlap{\kern 0.336in\lower\graphtemp\hbox to 0pt{\hss $C^*$\hss}}%
    \graphtemp=.5ex\advance\graphtemp by 0.269in
    \rlap{\kern 0.807in\lower\graphtemp\hbox to 0pt{\hss $z$\hss}}%
    \hbox{\vrule depth1.008in width0pt height 0pt}%
    \kern 2.000in
  }%
}%
}
\vspace{-1pc}
\caption{Case 1b of Theorem~\ref{THMA} when $d^+(V^*,v)=1$.\label{fig1b1} }
\end{center}
\end{figure}
\vspace{-1pc}

Therefore, we may assume $d^+(V^*,v)=0$, so $v$ dominates $V^*$ (see
Figures~\ref{fig1b2} and~\ref{fig1b3}).  By Lemma~\ref{lem1}, $T[V^*]$ contains
a cycle $C'$ of length $q-1$ such that the vertex $y$ of $V^*-V(C')$ has at
least two predecessors in $C'$.  Let $V'=V(C')$.

If $T[V'\cup\{w\}]$ is strong for some $w\in V(B)$, then let $\widehat C$ be
a spanning cycle in $T[V'\cup\{w\}]$ (outside $\cF$ only $w$ is used).  Since
$V(P)-\{v\}$ dominates $v$, we can let $P'$ be a path through all of
$V(P)-\{w,v\}$ followed by $v$ and $y$.  Since $d^+(V',y)\ge2$,
the degree fact applies to $P'$.

\begin{figure}[hbt]
\begin{center}
\gpic{
\expandafter\ifx\csname graph\endcsname\relax \csname newbox\endcsname\graph\fi
\expandafter\ifx\csname graphtemp\endcsname\relax \csname newdimen\endcsname\graphtemp\fi
\setbox\graph=\vtop{\vskip 0pt\hbox{%
    \special{pn 11}%
    \special{ar 1311 840 252 118 0 6.28319}%
    \special{pn 8}%
    \special{pa 134 840}%
    \special{pa 689 840}%
    \special{fp}%
    \special{sh 1.000}%
    \special{pa 555 807}%
    \special{pa 689 840}%
    \special{pa 555 874}%
    \special{pa 555 807}%
    \special{fp}%
    \special{pa 689 840}%
    \special{pa 1059 840}%
    \special{fp}%
    \special{pa 471 168}%
    \special{pa 1076 168}%
    \special{fp}%
    \special{sh 1.000}%
    \special{pa 941 134}%
    \special{pa 1076 168}%
    \special{pa 941 202}%
    \special{pa 941 134}%
    \special{fp}%
    \special{pa 1076 168}%
    \special{pa 1479 168}%
    \special{fp}%
    \special{pn 11}%
    \special{pa 1563 840}%
    \special{pa 1815 840}%
    \special{fp}%
    \graphtemp=.5ex\advance\graphtemp by 0.168in
    \rlap{\kern 1.479in\lower\graphtemp\hbox to 0pt{\hss $\bu$\hss}}%
    \graphtemp=.5ex\advance\graphtemp by 0.757in
    \rlap{\kern 1.133in\lower\graphtemp\hbox to 0pt{\hss $\bu$\hss}}%
    \graphtemp=.5ex\advance\graphtemp by 0.840in
    \rlap{\kern 1.815in\lower\graphtemp\hbox to 0pt{\hss $\bu$\hss}}%
    \graphtemp=.5ex\advance\graphtemp by 0.168in
    \rlap{\kern 0.807in\lower\graphtemp\hbox to 0pt{\hss $\bu$\hss}}%
    \graphtemp=.5ex\advance\graphtemp by 0.168in
    \rlap{\kern 0.807in\lower\graphtemp\hbox to 0pt{\hss $\bu$\hss}}%
    \graphtemp=.5ex\advance\graphtemp by 0.168in
    \rlap{\kern 1.479in\lower\graphtemp\hbox to 0pt{\hss $\bu$\hss}}%
    \graphtemp=.5ex\advance\graphtemp by 0.168in
    \rlap{\kern 1.815in\lower\graphtemp\hbox to 0pt{\hss $\bu$\hss}}%
    \graphtemp=.5ex\advance\graphtemp by 0.874in
    \rlap{\kern 1.429in\lower\graphtemp\hbox to 0pt{\hss $B$\hss}}%
    \special{pn 8}%
    \special{pa 1815 840}%
    \special{pa 1815 437}%
    \special{fp}%
    \special{sh 1.000}%
    \special{pa 1782 571}%
    \special{pa 1815 437}%
    \special{pa 1849 571}%
    \special{pa 1782 571}%
    \special{fp}%
    \special{pa 1815 437}%
    \special{pa 1815 168}%
    \special{fp}%
    \special{pn 28}%
    \special{ar 975 1414 1345 1345 -1.955193 -1.186400}%
    \special{pn 11}%
    \special{pa 807 168}%
    \special{pa 1479 168}%
    \special{fp}%
    \special{pn 8}%
    \special{pa 1133 757}%
    \special{pa 1340 404}%
    \special{fp}%
    \special{sh 1.000}%
    \special{pa 1243 503}%
    \special{pa 1340 404}%
    \special{pa 1301 537}%
    \special{pa 1243 503}%
    \special{fp}%
    \special{pa 1340 404}%
    \special{pa 1479 168}%
    \special{fp}%
    \special{pa 1025 1008}%
    \special{pa 1866 1008}%
    \special{pa 1866 672}%
    \special{pa 1025 672}%
    \special{pa 1025 1008}%
    \special{da 0.067}%
    \graphtemp=.5ex\advance\graphtemp by 0.840in
    \rlap{\kern 2.000in\lower\graphtemp\hbox to 0pt{\hss $S_1$\hss}}%
    \special{pn 28}%
    \special{pa 807 168}%
    \special{pa 1133 757}%
    \special{fp}%
    \special{pn 8}%
    \special{pa 937 404}%
    \special{pa 1002 522}%
    \special{fp}%
    \special{sh 1.000}%
    \special{pa 967 388}%
    \special{pa 1002 522}%
    \special{pa 908 420}%
    \special{pa 967 388}%
    \special{fp}%
    \special{pn 28}%
    \special{pa 1133 757}%
    \special{pa 1479 168}%
    \special{fp}%
    \special{pn 8}%
    \special{pa 1271 522}%
    \special{pa 1340 404}%
    \special{fp}%
    \special{sh 1.000}%
    \special{pa 1243 503}%
    \special{pa 1340 404}%
    \special{pa 1301 537}%
    \special{pa 1243 503}%
    \special{fp}%
    \special{pn 28}%
    \special{pa 471 168}%
    \special{pa 807 168}%
    \special{fp}%
    \special{pa 1479 168}%
    \special{pa 471 168}%
    \special{fp}%
    \graphtemp=.5ex\advance\graphtemp by 0.263in
    \rlap{\kern 1.910in\lower\graphtemp\hbox to 0pt{\hss $y$\hss}}%
    \graphtemp=.5ex\advance\graphtemp by 0.819in
    \rlap{\kern 1.228in\lower\graphtemp\hbox to 0pt{\hss $w$\hss}}%
    \graphtemp=.5ex\advance\graphtemp by 0.941in
    \rlap{\kern 1.815in\lower\graphtemp\hbox to 0pt{\hss $v$\hss}}%
    \graphtemp=.5ex\advance\graphtemp by 0.840in
    \rlap{\kern 0.000in\lower\graphtemp\hbox to 0pt{\hss $P$\hss}}%
    \graphtemp=.5ex\advance\graphtemp by 0.168in
    \rlap{\kern 0.336in\lower\graphtemp\hbox to 0pt{\hss $C'$\hss}}%
    \graphtemp=.5ex\advance\graphtemp by 0.230in
    \rlap{\kern 0.712in\lower\graphtemp\hbox to 0pt{\hss $~$\hss}}%
    \hbox{\vrule depth1.008in width0pt height 0pt}%
    \kern 2.000in
  }%
}%
}
\vspace{-1pc}
\caption{Case 1b of Theorem~\ref{THMA} when $d^+(V^*,v)=0$
and $T[V'\cup\{w\}]$ is strong.\label{fig1b2} }
\end{center}
\end{figure}
\vspace{-1pc}

In the remaining case, every vertex of $B$ dominates or is dominated by $V'$.
If $V'$ dominates some vertex of $B$, then $d^+(V^*,S_1)\ge q-1$, but
$d^+(V^*,S_1)\le \frac{|S_1|+1}{2q-2}|S_1|$ by Lemma~\ref{lem4}(a).  Since
$q\!-\!1\le \frac{|S_1|+1}{2q-2}|S_1|\le \FR{q^2+q}{2q-2}$ requires $q\!<\!5$,
we conclude that $V(B)$ dominates $V'$.

Since also $d^+(V^*,v)=0$, now $S_1$ dominates $V'$ (see Figure~\ref{fig1b3}).
The remaining argument is
similar to Case 1a.  Since $q-1\ge4$, by Lemma~\ref{lem1} there is a cycle
$C''$ of length $q-2$ in $T[V']$ such that the vertex $y'$ of $V'-V(C'')$ has
at least two predecessors in $V(C'')$.  Now $y$ and $y'$ each have at
at least two predecessors in $V^*$.  Let $zz'$ be the edge joining $y$ and $y'$.

Since $S_1$ dominates $V'$, any vertex of $C''$ has a successor in $S_2$, by
Lemma~\ref{lem4}(d); let $w$ be one such successor.  Now $T[V(C'')\cup\{w,v\}]$
is strong and has a spanning cycle $\widehat C$ of length $q$.  Form ${\cF'}$
from ${\cF}$ by replacing $C^*$ with $\widehat C$ (note that $\cF'$ uses only
$w$ and $v$ outside ${\cF}$).

Using Lemma~\ref{lem4}(b), $d^+(V(B),z)\ge1$; let $x$ be a predecessor of $z$
in $V(B)$.  Since $S_2$ dominates $S_1$, we can build a path $P'$ that starts
with all of $S_2-\{w\}$ (in some order), then visits all of $V(B)$ ending with
$x$, and finally follows $\la x,z,z'\ra$.  Since $v$ dominates $V^*$, vertex
$z'$ has at least two predecessors in $\widehat C$, and the degree fact applies.

\begin{figure}[hbt]
\begin{center}
\gpic{
\expandafter\ifx\csname graph\endcsname\relax \csname newbox\endcsname\graph\fi
\expandafter\ifx\csname graphtemp\endcsname\relax \csname newdimen\endcsname\graphtemp\fi
\setbox\graph=\vtop{\vskip 0pt\hbox{%
    \special{pn 11}%
    \special{ar 1279 787 246 115 0 6.28319}%
    \special{pn 8}%
    \special{pa 459 787}%
    \special{pa 803 787}%
    \special{fp}%
    \special{sh 1.000}%
    \special{pa 672 754}%
    \special{pa 803 787}%
    \special{pa 672 820}%
    \special{pa 672 754}%
    \special{fp}%
    \special{pa 803 787}%
    \special{pa 1033 787}%
    \special{fp}%
    \special{pn 11}%
    \special{pa 1525 787}%
    \special{pa 1770 787}%
    \special{fp}%
    \graphtemp=.5ex\advance\graphtemp by 0.787in
    \rlap{\kern 0.459in\lower\graphtemp\hbox to 0pt{\hss $\bu$\hss}}%
    \graphtemp=.5ex\advance\graphtemp by 0.706in
    \rlap{\kern 1.105in\lower\graphtemp\hbox to 0pt{\hss $\bu$\hss}}%
    \graphtemp=.5ex\advance\graphtemp by 0.787in
    \rlap{\kern 1.770in\lower\graphtemp\hbox to 0pt{\hss $\bu$\hss}}%
    \graphtemp=.5ex\advance\graphtemp by 0.131in
    \rlap{\kern 0.459in\lower\graphtemp\hbox to 0pt{\hss $\bu$\hss}}%
    \graphtemp=.5ex\advance\graphtemp by 0.131in
    \rlap{\kern 0.459in\lower\graphtemp\hbox to 0pt{\hss $\bu$\hss}}%
    \graphtemp=.5ex\advance\graphtemp by 0.131in
    \rlap{\kern 0.459in\lower\graphtemp\hbox to 0pt{\hss $\bu$\hss}}%
    \graphtemp=.5ex\advance\graphtemp by 0.295in
    \rlap{\kern 1.443in\lower\graphtemp\hbox to 0pt{\hss $\bu$\hss}}%
    \graphtemp=.5ex\advance\graphtemp by 0.131in
    \rlap{\kern 1.770in\lower\graphtemp\hbox to 0pt{\hss $\bu$\hss}}%
    \special{pn 28}%
    \special{ar 623 1347 1311 1311 -1.955193 -1.186400}%
    \special{ar 1115 -154 1148 1148 0.962551 2.179042}%
    \special{pn 8}%
    \special{pa 1105 706}%
    \special{pa 1307 459}%
    \special{fp}%
    \special{sh 1.000}%
    \special{pa 1199 540}%
    \special{pa 1307 459}%
    \special{pa 1250 581}%
    \special{pa 1199 540}%
    \special{fp}%
    \special{pa 1307 459}%
    \special{pa 1443 295}%
    \special{fp}%
    \special{pa 1443 295}%
    \special{pa 1639 197}%
    \special{fp}%
    \special{sh 1.000}%
    \special{pa 1507 226}%
    \special{pa 1639 197}%
    \special{pa 1537 285}%
    \special{pa 1507 226}%
    \special{fp}%
    \special{pa 1639 197}%
    \special{pa 1770 131}%
    \special{fp}%
    \special{pn 28}%
    \special{pa 459 131}%
    \special{pa 459 787}%
    \special{fp}%
    \special{pn 8}%
    \special{pa 459 393}%
    \special{pa 459 525}%
    \special{fp}%
    \special{sh 1.000}%
    \special{pa 492 393}%
    \special{pa 459 525}%
    \special{pa 426 393}%
    \special{pa 492 393}%
    \special{fp}%
    \special{pn 28}%
    \special{pa 1770 787}%
    \special{pa 1115 131}%
    \special{fp}%
    \special{pn 8}%
    \special{pa 1508 525}%
    \special{pa 1377 393}%
    \special{fp}%
    \special{sh 1.000}%
    \special{pa 1447 509}%
    \special{pa 1377 393}%
    \special{pa 1493 463}%
    \special{pa 1447 509}%
    \special{fp}%
    \special{pn 28}%
    \special{pa 131 131}%
    \special{pa 1115 131}%
    \special{fp}%
    \special{pn 8}%
    \special{pa 525 131}%
    \special{pa 721 131}%
    \special{fp}%
    \special{sh 1.000}%
    \special{pa 590 98}%
    \special{pa 721 131}%
    \special{pa 590 164}%
    \special{pa 590 98}%
    \special{fp}%
    \special{pn 11}%
    \special{pa 131 787}%
    \special{pa 459 787}%
    \special{fp}%
    \graphtemp=.5ex\advance\graphtemp by 0.787in
    \rlap{\kern 0.000in\lower\graphtemp\hbox to 0pt{\hss $P$\hss}}%
    \graphtemp=.5ex\advance\graphtemp by 0.131in
    \rlap{\kern 0.000in\lower\graphtemp\hbox to 0pt{\hss $C''$\hss}}%
    \graphtemp=.5ex\advance\graphtemp by 0.131in
    \rlap{\kern 1.902in\lower\graphtemp\hbox to 0pt{\hss $z'$\hss}}%
    \graphtemp=.5ex\advance\graphtemp by 0.820in
    \rlap{\kern 1.393in\lower\graphtemp\hbox to 0pt{\hss $B$\hss}}%
    \special{pn 8}%
    \special{pa 951 951}%
    \special{pa 1869 951}%
    \special{pa 1869 623}%
    \special{pa 951 623}%
    \special{pa 951 951}%
    \special{da 0.066}%
    \graphtemp=.5ex\advance\graphtemp by 0.787in
    \rlap{\kern 2.000in\lower\graphtemp\hbox to 0pt{\hss $S_1$\hss}}%
    \graphtemp=.5ex\advance\graphtemp by 0.918in
    \rlap{\kern 0.459in\lower\graphtemp\hbox to 0pt{\hss $w$\hss}}%
    \graphtemp=.5ex\advance\graphtemp by 0.202in
    \rlap{\kern 1.383in\lower\graphtemp\hbox to 0pt{\hss $z$\hss}}%
    \graphtemp=.5ex\advance\graphtemp by 0.766in
    \rlap{\kern 1.165in\lower\graphtemp\hbox to 0pt{\hss $x$\hss}}%
    \graphtemp=.5ex\advance\graphtemp by 0.885in
    \rlap{\kern 1.770in\lower\graphtemp\hbox to 0pt{\hss $v$\hss}}%
    \hbox{\vrule depth0.993in width0pt height 0pt}%
    \kern 2.000in
  }%
}%
}
\vspace{-1pc}
\caption{Case 1b of Theorem~\ref{THMA} when $d^+(V^*,v)=0$
and $V(B)$ dominates $V'$.\label{fig1b3} }
\end{center}
\end{figure}
\vspace{-1pc}

\medskip

{\bf Case 2:}
{\it $l>0$, so $T'$ has a cycle through $v$, the longest having length $l$.}
Let $C$ be such a cycle of length $l$.  We find a new family ${\cF'}$ and path
$P'$ outside it with a longer cycle through the last vertex of $P'$.
We may assume $l<q$, since otherwise pancyclicity yields the desired $q$-cycle.
Let $S_1=V(C)$ and $S_2=V(T')-S_1$.  If there is an edge from $S_1$ to $S_2$,
then $S_1$ and part of $P$ induce a strong tournament, which has a longer
spanning cycle (containing $v$).  Hence $S_2$ dominates $S_1$, and
Lemma~\ref{lem4} applies.

Let $C^*$ with vertex set $V^*$ be the cycle in ${\cF}$ guaranteed by
Lemma~\ref{lem4}.  Since $q\ge5$, by Lemma~\ref{lem2} $T[V^*]$ contains a cycle
$C''$ of length $q-2$ and an edge $xy$ with $\{x,y\}=V^*-V(C'')$ such that $y$
has at least one predecessor in $V(C'')$.  Let $V''=V(C'')$.
By Lemma~\ref{lem4}(a),
$$
d^+(V'',S_1)\le d^+(V^*,S_1)\le \FR{|S_1|+1}{2q-2}|S_1|
\le \FR q{2q-2}(q-1)=\FR q2.
$$
Since $q/2<q-2$ when $q\ge5$, in $V''$ there is a vertex $z$ dominated by
$S_1$.  By Lemma~\ref{lem4}(d), $z$ has a successor $w$ in $S_2$.

%
%

\medskip
{\bf Case 2a.} {\it $|S_1|=3$} (see Figure~\ref{fig2a}).
Here $d^+(V^*,S_1)\le \FR 4{2q-2}\cdot3<2$, by Lemma~\ref{lem4}(a).
If $x$ or $y$ has a successor in $S_1$, then let $u'$ be this successor;
otherwise choose $u'\in S_1$ arbitrarily.  Let the cycle $C$ through $S_1$ be
$[u',x',u]$, so $x'x\in E(T)$.  Note that $T[V''\cup \{w,u\}]$ is strong, with
a spanning cycle $\widehat C$.  Form $\cF'$ from $\cF$ by replacing $C^*$
with $\widehat C$.

Let $P'$ follow a spanning path through $S_2-\{w\}$ and then $\la u',x',x,y\ra$.
If $yu'\in E(T)$, then we have the cycle $[y,u',x',x]$ through $y$.
Otherwise $d^+(y,S_1)=0$, and by Lemma~\ref{lem4}(d) $y$ has at least two
successors in $S_2$, which means it has one other than $w$.  In this case
$y$ lies on a cycle of length more than $l$ in $T[V(P')]$.

\begin{figure}[hbt]
\begin{center}
\gpic{
\expandafter\ifx\csname graph\endcsname\relax \csname newbox\endcsname\graph\fi
\expandafter\ifx\csname graphtemp\endcsname\relax \csname newdimen\endcsname\graphtemp\fi
\setbox\graph=\vtop{\vskip 0pt\hbox{%
    \graphtemp=.5ex\advance\graphtemp by 0.787in
    \rlap{\kern 0.459in\lower\graphtemp\hbox to 0pt{\hss $\bu$\hss}}%
    \graphtemp=.5ex\advance\graphtemp by 0.689in
    \rlap{\kern 1.344in\lower\graphtemp\hbox to 0pt{\hss $\bu$\hss}}%
    \graphtemp=.5ex\advance\graphtemp by 0.787in
    \rlap{\kern 1.770in\lower\graphtemp\hbox to 0pt{\hss $\bu$\hss}}%
    \graphtemp=.5ex\advance\graphtemp by 0.787in
    \rlap{\kern 1.049in\lower\graphtemp\hbox to 0pt{\hss $\bu$\hss}}%
    \graphtemp=.5ex\advance\graphtemp by 0.131in
    \rlap{\kern 0.459in\lower\graphtemp\hbox to 0pt{\hss $\bu$\hss}}%
    \graphtemp=.5ex\advance\graphtemp by 0.131in
    \rlap{\kern 0.459in\lower\graphtemp\hbox to 0pt{\hss $\bu$\hss}}%
    \graphtemp=.5ex\advance\graphtemp by 0.295in
    \rlap{\kern 1.443in\lower\graphtemp\hbox to 0pt{\hss $\bu$\hss}}%
    \graphtemp=.5ex\advance\graphtemp by 0.131in
    \rlap{\kern 1.770in\lower\graphtemp\hbox to 0pt{\hss $\bu$\hss}}%
    \special{pn 8}%
    \special{pa 459 787}%
    \special{pa 813 787}%
    \special{fp}%
    \special{sh 1.000}%
    \special{pa 682 754}%
    \special{pa 813 787}%
    \special{pa 682 820}%
    \special{pa 682 754}%
    \special{fp}%
    \special{pa 813 787}%
    \special{pa 1049 787}%
    \special{fp}%
    \special{pn 28}%
    \special{ar 623 1347 1311 1311 -1.955193 -1.186400}%
    \special{ar 1115 -154 1148 1148 0.962551 2.179042}%
    \special{pn 8}%
    \special{pa 1344 689}%
    \special{pa 1403 452}%
    \special{fp}%
    \special{sh 1.000}%
    \special{pa 1340 572}%
    \special{pa 1403 452}%
    \special{pa 1403 588}%
    \special{pa 1340 572}%
    \special{fp}%
    \special{pa 1403 452}%
    \special{pa 1443 295}%
    \special{fp}%
    \special{pa 1443 295}%
    \special{pa 1639 197}%
    \special{fp}%
    \special{sh 1.000}%
    \special{pa 1507 226}%
    \special{pa 1639 197}%
    \special{pa 1537 285}%
    \special{pa 1507 226}%
    \special{fp}%
    \special{pa 1639 197}%
    \special{pa 1770 131}%
    \special{fp}%
    \special{pn 28}%
    \special{pa 459 131}%
    \special{pa 459 787}%
    \special{fp}%
    \special{pn 8}%
    \special{pa 459 393}%
    \special{pa 459 525}%
    \special{fp}%
    \special{sh 1.000}%
    \special{pa 492 393}%
    \special{pa 459 525}%
    \special{pa 426 393}%
    \special{pa 492 393}%
    \special{fp}%
    \special{pn 28}%
    \special{pa 1770 787}%
    \special{pa 1115 131}%
    \special{fp}%
    \special{pn 8}%
    \special{pa 1508 525}%
    \special{pa 1377 393}%
    \special{fp}%
    \special{sh 1.000}%
    \special{pa 1447 509}%
    \special{pa 1377 393}%
    \special{pa 1493 463}%
    \special{pa 1447 509}%
    \special{fp}%
    \special{pn 28}%
    \special{pa 131 131}%
    \special{pa 1115 131}%
    \special{fp}%
    \special{pn 8}%
    \special{pa 525 131}%
    \special{pa 721 131}%
    \special{fp}%
    \special{sh 1.000}%
    \special{pa 590 98}%
    \special{pa 721 131}%
    \special{pa 590 164}%
    \special{pa 590 98}%
    \special{fp}%
    \special{pn 11}%
    \special{pa 131 787}%
    \special{pa 459 787}%
    \special{fp}%
    \special{pn 8}%
    \special{pa 1049 787}%
    \special{pa 1226 728}%
    \special{fp}%
    \special{sh 1.000}%
    \special{pa 1091 738}%
    \special{pa 1226 728}%
    \special{pa 1112 800}%
    \special{pa 1091 738}%
    \special{fp}%
    \special{pa 1226 728}%
    \special{pa 1344 689}%
    \special{fp}%
    \special{pa 1344 689}%
    \special{pa 1600 748}%
    \special{fp}%
    \special{sh 1.000}%
    \special{pa 1480 686}%
    \special{pa 1600 748}%
    \special{pa 1465 750}%
    \special{pa 1480 686}%
    \special{fp}%
    \special{pa 1600 748}%
    \special{pa 1770 787}%
    \special{fp}%
    \special{pa 1770 787}%
    \special{pa 1338 787}%
    \special{fp}%
    \special{sh 1.000}%
    \special{pa 1469 820}%
    \special{pa 1338 787}%
    \special{pa 1469 754}%
    \special{pa 1469 820}%
    \special{fp}%
    \special{pa 1338 787}%
    \special{pa 1049 787}%
    \special{fp}%
    \graphtemp=.5ex\advance\graphtemp by 0.787in
    \rlap{\kern 0.000in\lower\graphtemp\hbox to 0pt{\hss $P$\hss}}%
    \graphtemp=.5ex\advance\graphtemp by 0.131in
    \rlap{\kern 0.000in\lower\graphtemp\hbox to 0pt{\hss $C''$\hss}}%
    \graphtemp=.5ex\advance\graphtemp by 0.131in
    \rlap{\kern 1.902in\lower\graphtemp\hbox to 0pt{\hss $y$\hss}}%
    \special{pa 951 951}%
    \special{pa 1869 951}%
    \special{pa 1869 623}%
    \special{pa 951 623}%
    \special{pa 951 951}%
    \special{da 0.066}%
    \graphtemp=.5ex\advance\graphtemp by 0.787in
    \rlap{\kern 2.000in\lower\graphtemp\hbox to 0pt{\hss $S_1$\hss}}%
    \graphtemp=.5ex\advance\graphtemp by 0.918in
    \rlap{\kern 0.459in\lower\graphtemp\hbox to 0pt{\hss $w$\hss}}%
    \graphtemp=.5ex\advance\graphtemp by 0.885in
    \rlap{\kern 1.049in\lower\graphtemp\hbox to 0pt{\hss $u'$\hss}}%
    \graphtemp=.5ex\advance\graphtemp by 0.202in
    \rlap{\kern 1.383in\lower\graphtemp\hbox to 0pt{\hss $x$\hss}}%
    \graphtemp=.5ex\advance\graphtemp by 0.563in
    \rlap{\kern 1.252in\lower\graphtemp\hbox to 0pt{\hss $x'$\hss}}%
    \graphtemp=.5ex\advance\graphtemp by 0.885in
    \rlap{\kern 1.770in\lower\graphtemp\hbox to 0pt{\hss $u$\hss}}%
    \hbox{\vrule depth0.993in width0pt height 0pt}%
    \kern 2.000in
  }%
}%
}
\vspace{-1pc}
\caption{Case 2a of Theorem~\ref{THMA}.\label{fig2a} }
\end{center}
\end{figure}
\vspace{-1pc}

{\bf Case 2b.} {\it $|S_1|\ge4$} (see Figure~\ref{fig2b}).
By pancyclicity, $T[S_1]$ contains a cycle $B$ omitting one vertex $u$ of
$S_1$.  Since $l<q$, Lemma~\ref{lem4}(c) implies $d^+(u,V^*)\ge3$, yielding
an edge $uz'$ with $z'\in V''$.  Using the edge $zw$ from $V''$ to $S_2$
(obtained earlier), the path $\la z,w,u,z'\ra$ guarantees that
$T[V''\cup\{w,u\}]$ is strong and hence has a spanning $q$-cycle $\widehat C$.
Form ${\cF'}$ from ${\cF}$ by replacing $C^*$ with $\widehat{C}$.

By Lemma~\ref{lem4}(b), $x$ has a predecessor $x'$ in $V(B)$.  Build the path
$P'$ outside $\cF'$ by visiting all of $S_2-\{w\}$ (in some order), then all of
$V(B)$ ending with $x'$, and finally $\la x',x,y\ra$.  Since $y$ has at least
two predecessors in $V''\cup\{x\}$ and hence at most $q-3$ successors in $V^*$,
$$
d^+(y,S_1\cup S_2)\ge [(q-1)k-1]-q(k-2)-(q-3)=q-k+2\ge2.
$$
If $y$ has a successor outside $\{w,u\}$ in $S_1\cup S_2$, then with
$V(B)\cup\{x,y\}$ it induces a strong tournament of order at least $l+1$,
yielding the desired cycle through the last vertex of $P'$.

\begin{figure}[hbt]
\begin{center}
\gpic{
\expandafter\ifx\csname graph\endcsname\relax \csname newbox\endcsname\graph\fi
\expandafter\ifx\csname graphtemp\endcsname\relax \csname newdimen\endcsname\graphtemp\fi
\setbox\graph=\vtop{\vskip 0pt\hbox{%
    \special{pn 11}%
    \special{ar 1279 787 246 115 0 6.28319}%
    \special{pn 8}%
    \special{pa 459 787}%
    \special{pa 803 787}%
    \special{fp}%
    \special{sh 1.000}%
    \special{pa 672 754}%
    \special{pa 803 787}%
    \special{pa 672 820}%
    \special{pa 672 754}%
    \special{fp}%
    \special{pa 803 787}%
    \special{pa 1033 787}%
    \special{fp}%
    \special{pn 11}%
    \special{pa 1525 787}%
    \special{pa 1770 787}%
    \special{fp}%
    \graphtemp=.5ex\advance\graphtemp by 0.787in
    \rlap{\kern 0.459in\lower\graphtemp\hbox to 0pt{\hss $\bu$\hss}}%
    \graphtemp=.5ex\advance\graphtemp by 0.706in
    \rlap{\kern 1.105in\lower\graphtemp\hbox to 0pt{\hss $\bu$\hss}}%
    \graphtemp=.5ex\advance\graphtemp by 0.787in
    \rlap{\kern 1.770in\lower\graphtemp\hbox to 0pt{\hss $\bu$\hss}}%
    \graphtemp=.5ex\advance\graphtemp by 0.131in
    \rlap{\kern 0.459in\lower\graphtemp\hbox to 0pt{\hss $\bu$\hss}}%
    \graphtemp=.5ex\advance\graphtemp by 0.131in
    \rlap{\kern 0.459in\lower\graphtemp\hbox to 0pt{\hss $\bu$\hss}}%
    \graphtemp=.5ex\advance\graphtemp by 0.131in
    \rlap{\kern 0.459in\lower\graphtemp\hbox to 0pt{\hss $\bu$\hss}}%
    \graphtemp=.5ex\advance\graphtemp by 0.295in
    \rlap{\kern 1.443in\lower\graphtemp\hbox to 0pt{\hss $\bu$\hss}}%
    \graphtemp=.5ex\advance\graphtemp by 0.131in
    \rlap{\kern 1.770in\lower\graphtemp\hbox to 0pt{\hss $\bu$\hss}}%
    \special{pn 28}%
    \special{ar 623 1347 1311 1311 -1.955193 -1.186400}%
    \special{ar 1115 -154 1148 1148 0.962551 2.179042}%
    \special{pn 8}%
    \special{pa 1105 706}%
    \special{pa 1307 459}%
    \special{fp}%
    \special{sh 1.000}%
    \special{pa 1199 540}%
    \special{pa 1307 459}%
    \special{pa 1250 581}%
    \special{pa 1199 540}%
    \special{fp}%
    \special{pa 1307 459}%
    \special{pa 1443 295}%
    \special{fp}%
    \special{pa 1443 295}%
    \special{pa 1639 197}%
    \special{fp}%
    \special{sh 1.000}%
    \special{pa 1507 226}%
    \special{pa 1639 197}%
    \special{pa 1537 285}%
    \special{pa 1507 226}%
    \special{fp}%
    \special{pa 1639 197}%
    \special{pa 1770 131}%
    \special{fp}%
    \special{pn 28}%
    \special{pa 459 131}%
    \special{pa 459 787}%
    \special{fp}%
    \special{pn 8}%
    \special{pa 459 393}%
    \special{pa 459 525}%
    \special{fp}%
    \special{sh 1.000}%
    \special{pa 492 393}%
    \special{pa 459 525}%
    \special{pa 426 393}%
    \special{pa 492 393}%
    \special{fp}%
    \special{pn 28}%
    \special{pa 1770 787}%
    \special{pa 1115 131}%
    \special{fp}%
    \special{pn 8}%
    \special{pa 1508 525}%
    \special{pa 1377 393}%
    \special{fp}%
    \special{sh 1.000}%
    \special{pa 1447 509}%
    \special{pa 1377 393}%
    \special{pa 1493 463}%
    \special{pa 1447 509}%
    \special{fp}%
    \special{pn 28}%
    \special{pa 131 131}%
    \special{pa 1115 131}%
    \special{fp}%
    \special{pn 8}%
    \special{pa 525 131}%
    \special{pa 721 131}%
    \special{fp}%
    \special{sh 1.000}%
    \special{pa 590 98}%
    \special{pa 721 131}%
    \special{pa 590 164}%
    \special{pa 590 98}%
    \special{fp}%
    \special{pn 11}%
    \special{pa 131 787}%
    \special{pa 459 787}%
    \special{fp}%
    \graphtemp=.5ex\advance\graphtemp by 0.787in
    \rlap{\kern 0.000in\lower\graphtemp\hbox to 0pt{\hss $P$\hss}}%
    \graphtemp=.5ex\advance\graphtemp by 0.131in
    \rlap{\kern 0.000in\lower\graphtemp\hbox to 0pt{\hss $C''$\hss}}%
    \graphtemp=.5ex\advance\graphtemp by 0.131in
    \rlap{\kern 1.902in\lower\graphtemp\hbox to 0pt{\hss $y$\hss}}%
    \graphtemp=.5ex\advance\graphtemp by 0.820in
    \rlap{\kern 1.393in\lower\graphtemp\hbox to 0pt{\hss $B$\hss}}%
    \special{pn 8}%
    \special{pa 951 951}%
    \special{pa 1869 951}%
    \special{pa 1869 623}%
    \special{pa 951 623}%
    \special{pa 951 951}%
    \special{da 0.066}%
    \graphtemp=.5ex\advance\graphtemp by 0.787in
    \rlap{\kern 2.000in\lower\graphtemp\hbox to 0pt{\hss $S_1$\hss}}%
    \graphtemp=.5ex\advance\graphtemp by 0.918in
    \rlap{\kern 0.459in\lower\graphtemp\hbox to 0pt{\hss $w$\hss}}%
    \graphtemp=.5ex\advance\graphtemp by 0.202in
    \rlap{\kern 1.383in\lower\graphtemp\hbox to 0pt{\hss $x$\hss}}%
    \graphtemp=.5ex\advance\graphtemp by 0.766in
    \rlap{\kern 1.198in\lower\graphtemp\hbox to 0pt{\hss $x'$\hss}}%
    \graphtemp=.5ex\advance\graphtemp by 0.885in
    \rlap{\kern 1.770in\lower\graphtemp\hbox to 0pt{\hss $u$\hss}}%
    \hbox{\vrule depth0.993in width0pt height 0pt}%
    \kern 2.000in
  }%
}%
}
\vspace{-1pc}
\caption{Case 2b of Theorem~\ref{THMA}.\label{fig2b} }
\end{center}
\end{figure}
\vspace{-1pc}

Hence $w$ and $u$ must be the only successors of $y$ in $S_1\cup S_2$.  This
forces $d^+(V'',y)=1$.  If any vertex of $V''$ has a successor in $S_2$ other
than $w$, then we obtain $\cF'$ as above using that vertex instead of $w$.
Hence $w$ must be the only successor of vertices in $V''$, so
$d^+(V'',S_2)\le q-2$.  We already computed $d^+(V^*,S_1)\le q/2$, and one of
the edges counted is $yu$.

For any $\alpha\in V''$, using $k\le q$ we have
$$
d^+(\alpha,S_1\cup S_2)\ge [(q-1)k-1]-q(k-2)-d^+(\alpha,V^*)
\ge q-1-d^+(\alpha,V^*).
$$
When we sum over all $\alpha\in V''$, the last term counts all edges in
$T[V'']$, possibly $q-2$ edges from $V''$ to $x$, and one edge from $V''$ to
$y$.  Hence
$$
d^+(V'', S_1\cup S_2)\ge (q-2)(q-1)-\CH{q-2}2-(q-2)-1 =(q-2)\FR{q-1}2-1.
$$
The upper and lower bounds on $d^+(V'',S_1\cup S_2)$ require
$(q-2)(q-1)/2-1\le 3q/2-3$, but this inequality requires $q<5$.
Hence we obtain the desired improvement $\cF'$.


\medskip
In all cases we have improved the family $\cF$ as desired.
\end{proof}

\begin{figure}
  \centering
\end{figure}


\section{The Case of Large $k$}\label{SecB}

For the proof of Theorem~\ref{THMB}, we will use Theorem~\ref{THMA}
as a basis for induction on $k$.  The two theorems have the same
conclusion.

\begin{theorem}\label{THMB}
Given $k,q\in\NN$ with $q \geq 5$, let $T$ be a tournament with
$\delta^+(T)\ge(q-1)k-1$.  For any family ${\cF}$ of $k-1$ disjoint $q$-cycles
in $T$, there is a family of $k$ disjoint $q$-cycles in $T$ whose union has at
most $3q-6$ vertices outside the cycles in ${\cF}$.
\end{theorem}

We first discuss the basic set-up for the argument, defining notation to be
used throughout the proof.  We call the desired family an {\it extension} of
${\cF}$; finding it is {\it extending} ${\cF}$.
By Theorem \ref{THMA}, we may assume $k\ge q+1$.  In the tournament $T$,
consider a family ${\cF}$ of $k-1$ disjoint $q$-cycles in $T$.  We may assume
that the tournament $T'$ given by deleting the vertices covered by ${\cF}$
contains no cycle with length at least $q$; otherwise we have the desired
extension.

Let $P$ be a spanning path in $T'$, listed as $\la \VEC ul1\ra$.
Since $\delta^+(T)\ge(q-1)k-1$ implies $|V(T)|\ge2(q-1)k-1$, we have
$|V(T')|\ge 2(q-1)k-1-q(k-1)=(q-2)(k+1)+1$.
Since $k \geq 3$, we conclude $l=|V(T')|\ge4q-7$.

Partition $V(P)$ into $\{U_1,S,U_2\}$ by letting $U_1=\{\VEC u1{q}\}$,
$S = \{\VEC u{q+1}{4q-11}\}$, and $U_2 = V(P)-(U_1\cup S)$.  All three sets are
nonempty, with $|S|=3q-11$ and $|U_2|\ge4$.  Also, since $T'$ contains no
$q$-cycle, the edge joining $u_j$ and $u_i$ is oriented as $u_{j}u_{i}$ when
$j-i\ge q-1$.  Hence $U_2$ dominates $U_1$.

We aim to find a value $t\in\{1,2\}$ such that we can replace $t$ cycles in
${\cF}$ with $t+1$ cycles of length $q$ using at most $3q-6$ vertices of $T'$.
This will complete the proof.

Let $X$ and $Y$ be two disjoint sets of vertices in $T$.  We say that there is
an \emph{$r$-matching from $X$ to $Y$} if the set of edges with tail in $X$
and head in $Y$ contains $r$ edges with no common endpoints.  In order to
guarantee the existence of desired matchings, we will use the famous
K\"onig--Egerv\'ary Theorem (K\"onig~\cite{Kon}, Egerv\'ary~\cite{Ege}),
phrased for bipartite digraphs with all edges directed from one part to the
other.

\begin{lemma}[\cite{Ege,Kon}]\label{L1}
If there is no $r$-matching from $X$ to $Y$, then $X \cup Y$ contains a set of
at most $r-1$ vertices whose deletion eliminates all edges from $X$ to $Y$.
\end{lemma}

For the proof of Theorem \ref{THMB}, we need a number of additional lemmas.
The first is a standard application of the K\"onig--Egerv\'ary Theorem,
which we will apply with various values of the parameters.

\begin{lemma}\label{ClaimA}
Let $X$ and $Y$ be disjoint vertex sets in $T$, with $s=\min\{|X|,|Y|\}$
and $t=\max\{|X|,|Y|\}$.  If $d^+(X,Y)>(r-1)t$, where $1\le r\le s$, then $T$
contains an $r$-matching from $X$ to $Y$.
\end{lemma}
\begin{proof}
Since $r-1$ vertices cover at most $(r-1)t$ edges, the K\"onig--Egerv\'ary
Theorem implies that the desired matching exists.
\end{proof}

\begin{lemma}\label{ClaimB}
Let $C$ be a $q$-cycle in ${\cF}$.  If there is a vertex $v \in V(C)$ with at
least $3q-6$ successors in $U_2$, each having at least two successors in $C$,
then there is an extension of ${\cF}$.
\end{lemma}

\begin{proof}
Let $W\esub U_2$ be such a set of $3q-6$ successors of $v$.
Since $T[V(C)]$ is strong, by Moon's Theorem it has a $(q-1)$-cycle $C'$
containing $v$, omitting one vertex $u$ of $C$.  Since each vertex $w\in W$
has a successor in $C$ other than $u$, the subtournament $T[V(C')\cup\{w\}]$
is strong and has a spanning $q$-cycle (see Figure~\ref{fig54}).

Let $T'=T-V(C')$ and $\cF_0= {\cF} -\{C\}$.  Since $T'$ omits only $q-1$
vertices, $\delta^+(T')\ge (q-1)(k-1)-1$, and $\cF_0$ is a family of $k-2$
cycles of length $q$ in $T'$.  Using the induction hypothesis, we can extend
$\cF_0$ to a family $\widehat{{\cF}}$ of $k-1$ cycles of length $q$ in $T'$
using at most $3q-6$ new vertices.  Since $|W\cup\{u\}|=3q-5$, some vertex in
$W\cup\{u\}$ is not used by $\widehat{{\cF}}$.

If $u$ is not used, then adding $C$ to $\widehat{{\cF}}$ completes the
desired extension.  If $u$ is used, then at most $3q-7$ vertices not in
${\cF}$ are used in $\widehat{{\cF}}$.  In this case, some vertex $w\in W$ is
not used, and a spanning cycle in $T[V(C')\cup\{w\}]$ completes the extension
$\cF'$ using a total of at most $3q-6$ vertices not in ${\cF}$.
\end{proof}

\begin{figure}[hbt]
\begin{center}
\gpic{
\expandafter\ifx\csname graph\endcsname\relax \csname newbox\endcsname\graph\fi
\expandafter\ifx\csname graphtemp\endcsname\relax \csname newdimen\endcsname\graphtemp\fi
\setbox\graph=\vtop{\vskip 0pt\hbox{%
    \graphtemp=.5ex\advance\graphtemp by 0.738in
    \rlap{\kern 0.800in\lower\graphtemp\hbox to 0pt{\hss $\bu$\hss}}%
    \graphtemp=.5ex\advance\graphtemp by 0.738in
    \rlap{\kern 0.800in\lower\graphtemp\hbox to 0pt{\hss $\bu$\hss}}%
    \graphtemp=.5ex\advance\graphtemp by 0.738in
    \rlap{\kern 2.277in\lower\graphtemp\hbox to 0pt{\hss $\bu$\hss}}%
    \graphtemp=.5ex\advance\graphtemp by 0.123in
    \rlap{\kern 1.108in\lower\graphtemp\hbox to 0pt{\hss $\bu$\hss}}%
    \graphtemp=.5ex\advance\graphtemp by 0.123in
    \rlap{\kern 1.108in\lower\graphtemp\hbox to 0pt{\hss $\bu$\hss}}%
    \graphtemp=.5ex\advance\graphtemp by 0.123in
    \rlap{\kern 1.108in\lower\graphtemp\hbox to 0pt{\hss $\bu$\hss}}%
    \graphtemp=.5ex\advance\graphtemp by 0.123in
    \rlap{\kern 2.031in\lower\graphtemp\hbox to 0pt{\hss $\bu$\hss}}%
    \special{pn 8}%
    \special{pa 185 738}%
    \special{pa 1440 738}%
    \special{fp}%
    \special{sh 1.000}%
    \special{pa 1317 708}%
    \special{pa 1440 738}%
    \special{pa 1317 769}%
    \special{pa 1317 708}%
    \special{fp}%
    \special{pa 1440 738}%
    \special{pa 2277 738}%
    \special{fp}%
    \special{pn 28}%
    \special{ar 1262 1264 1231 1231 -1.955193 -1.186400}%
    \special{pa 800 123}%
    \special{pa 1723 123}%
    \special{fp}%
    \special{pn 8}%
    \special{pa 1169 123}%
    \special{pa 1354 123}%
    \special{fp}%
    \special{sh 1.000}%
    \special{pa 1231 92}%
    \special{pa 1354 123}%
    \special{pa 1231 154}%
    \special{pa 1231 92}%
    \special{fp}%
    \special{pn 28}%
    \special{pa 800 738}%
    \special{pa 1569 123}%
    \special{fp}%
    \special{pn 8}%
    \special{pa 1108 492}%
    \special{pa 1262 369}%
    \special{fp}%
    \special{sh 1.000}%
    \special{pa 1146 422}%
    \special{pa 1262 369}%
    \special{pa 1185 470}%
    \special{pa 1146 422}%
    \special{fp}%
    \special{pn 28}%
    \special{pa 1108 123}%
    \special{pa 800 738}%
    \special{fp}%
    \special{pn 8}%
    \special{pa 985 369}%
    \special{pa 923 492}%
    \special{fp}%
    \special{sh 1.000}%
    \special{pa 1006 396}%
    \special{pa 923 492}%
    \special{pa 951 368}%
    \special{pa 1006 396}%
    \special{fp}%
    \special{pa 800 738}%
    \special{pa 1538 369}%
    \special{fp}%
    \special{sh 1.000}%
    \special{pa 1415 397}%
    \special{pa 1538 369}%
    \special{pa 1442 452}%
    \special{pa 1415 397}%
    \special{fp}%
    \special{pa 1538 369}%
    \special{pa 2031 123}%
    \special{fp}%
    \special{pa 1754 923}%
    \special{pa 2369 923}%
    \special{pa 2369 615}%
    \special{pa 1754 615}%
    \special{pa 1754 923}%
    \special{da 0.062}%
    \special{pa 154 923}%
    \special{pa 1077 923}%
    \special{pa 1077 615}%
    \special{pa 154 615}%
    \special{pa 154 923}%
    \special{da 0.062}%
    \special{pa 1108 923}%
    \special{pa 1723 923}%
    \special{pa 1723 615}%
    \special{pa 1108 615}%
    \special{pa 1108 923}%
    \special{da 0.062}%
    \special{pa 431 892}%
    \special{pa 985 892}%
    \special{pa 985 646}%
    \special{pa 431 646}%
    \special{pa 431 892}%
    \special{da 0.062}%
    \graphtemp=.5ex\advance\graphtemp by 0.769in
    \rlap{\kern 0.000in\lower\graphtemp\hbox to 0pt{\hss $P$\hss}}%
    \graphtemp=.5ex\advance\graphtemp by 0.123in
    \rlap{\kern 0.677in\lower\graphtemp\hbox to 0pt{\hss $C'$\hss}}%
    \graphtemp=.5ex\advance\graphtemp by 1.046in
    \rlap{\kern 1.415in\lower\graphtemp\hbox to 0pt{\hss $S$\hss}}%
    \graphtemp=.5ex\advance\graphtemp by 1.046in
    \rlap{\kern 2.062in\lower\graphtemp\hbox to 0pt{\hss $U_1$\hss}}%
    \graphtemp=.5ex\advance\graphtemp by 0.795in
    \rlap{\kern 0.887in\lower\graphtemp\hbox to 0pt{\hss $w$\hss}}%
    \graphtemp=.5ex\advance\graphtemp by 0.210in
    \rlap{\kern 0.990in\lower\graphtemp\hbox to 0pt{\hss $v$\hss}}%
    \graphtemp=.5ex\advance\graphtemp by 0.123in
    \rlap{\kern 2.154in\lower\graphtemp\hbox to 0pt{\hss $u$\hss}}%
    \graphtemp=.5ex\advance\graphtemp by 0.795in
    \rlap{\kern 0.887in\lower\graphtemp\hbox to 0pt{\hss $~$\hss}}%
    \graphtemp=.5ex\advance\graphtemp by 0.831in
    \rlap{\kern 2.277in\lower\graphtemp\hbox to 0pt{\hss $u_1$\hss}}%
    \graphtemp=.5ex\advance\graphtemp by 1.046in
    \rlap{\kern 0.615in\lower\graphtemp\hbox to 0pt{\hss $U_2$\hss}}%
    \graphtemp=.5ex\advance\graphtemp by 0.856in
    \rlap{\kern 0.518in\lower\graphtemp\hbox to 0pt{\hss $W$\hss}}%
    \hbox{\vrule depth1.046in width0pt height 0pt}%
    \kern 2.400in
  }%
}%
}
\vspace{-1pc}
\caption{Figure for Lemma~\ref{ClaimB}.\label{fig54} }
\end{center}
\end{figure}
\vspace{-1pc}

The need to find an unused vertex in $W\cup\{u\}$ in the preceding proof is the
reason we limit the number of new vertices in Theorem~\ref{THMA} and in the
induction hypothesis there.

We use Lemma~\ref{ClaimB} to prove the next two lemmas.

\begin{lemma}\label{ClaimC}
Let $C$ be a $q$-cycle in ${\cF}$.  Suppose that $\cF$ has no extension.

(i) If $d^{+}(C , U_{2}) \geq 3q-6$, then $T$ contains a 2-matching from $C$ to
$U_{2}$.

(ii) If $d^{+}(C , U_{2})\geq 6q-13$, then $T$ contains a 3-matching from $C$ to
$U_{2}$.
\end{lemma}

\begin{proof}
(i) If there is no such $2$-matching, then by Lemma~\ref{L1} one vertex covers
all the edges from $V(C)$ to $U_2$.  Such a vertex $v$ can only be in $C$, and
there is no other edge from $V(C)$ to $U_2$.
Since $q\ge3$, each of the $3q-6$ successors of $v$ in $U_2$ has at least two
successors in $V(C)$.  By Lemma~\ref{ClaimB}, we have an extension of $\cF$.

(ii) If there is no such $3$-matching, then by Lemma~\ref{L1} two vertices
$u$ and $v$ cover all the edges from $V(C)$ to $U_2$, of which there are at
least $6q-13$.  If $u$ and $v$ are both in $V(C)$, then one has at least $3q-6$
successors in $U_2$, each of which has at least two successors in $C$, since
$q\ge4$.  Otherwise, name $u$ and $v$ with $u\in U_2$ and $v\in V(C)$; now $v$
has at least $5q-14$ successors in $U_2$ other than $u$.  These vertices have
no predecessors in $V(C)$ other than $v$; hence they have at least two
successors in $V(C)$.  Since $5q-14\ge 3q-6$ when $q\ge4$, in either
case Lemma~\ref{ClaimB} applies to guarantee an extension of $\cF$.
\end{proof}

\begin{lemma}\label{ClaimD}
Let $C$ be a $q$-cycle in ${\cF}$.  If $T$ contains

(i) a $q$-matching from $U_{1}$ to $V(C)$ and a $2$-matching from $V(C)$ to
$U_{2}$, or

(ii) a $(q-1)$-matching from $U_{1}$ to $V(C)$ and a $3$-matching from $V(C)$
to $U_{2}$,

then there is an extension of ${\cF}$.
\end{lemma}

\begin{proof}
We will obtain two disjoint cycles of length at least $q$ in $T[V(C)\cup V(P)]$,
where $P$ is a spanning path through the vertices outside $\cF$.  By
pancyclicity, the subtournament induced by the vertices of each such cycle
contains a $q$-cycle.  Since $T'$ induces no cycle of length at least $q$, each
of the two new cycles replacing $C$ contains at most $q-1$ new vertices.

Since $q\ge5$, we have $3q-11\ge 2q-6$.  Hence $|S|\ge2q-6$.
Since $V(P)$ induces no cycle of length at least $q$, any edge joining
two vertices of $P$ with at least $q-2$ vertices between them on $P$ is
directed from the earlier to the later vertex.

Let $M$ and $M'$ be the given matchings from $U_1$ to $V(C)$ and from
$V(C)$ to $U_2$, respectively.  We prove (i) and (ii) together.
In either case, let $u$ be the last vertex of $P$ matched into $C$
from $U_1$ by $M$; note that $u\in \{u_2,u_1\}$.  After following the edge from
$u$ to $V(C)$, let $v$ be the vertex matched into $U_2$ by $M'$ that is reached
first when continuing along $C$, and let $vw$ be this edge of $M'$.  Let $Q$ be
the path thus followed, from $u$ via $M$, along $C$, ending with $vw$.

Now choose $yz\in M'-\{vw\}$ so that $y$ is also the head of an edge in $M$
(under (ii), more than one edge of $M'$ remains, but then at most one vertex of
$C$ is not covered by $M$.)  Say that a vertex of $U_1$
{\it leads to} $y$ if it is matched by $M$ into the path along $C$ that starts
with the successor of $v$ on $C$ and ends at $y$.  If $z$ is closer to $S$ than
$w$ along $P$, then let $x$ be the highest-indexed vertex of $U_1$ (closest to
$S$) that leads to $y$.  Otherwise, let $x$ be the lowest-indexed vertex of of
$U_1$ that leads to $y$.  Let $R$ be the path leaving $x$ via $M$, then along
$C$ to $y$, ending with $yz$.
We want to form two cycles of length at least $q$ in $T[V(C)\cup V(P)]$
by adding vertices of $P$ to $Q$ or $R$.  Let $P[a,b]$ or $C[a,b]$ denote the
$a,b$-path along $P$ or $C$.

\smallskip
{\bf Case 1:} {\it $z$ is closer to $S$ than $w$ along $P$, so $x$ is
the highest-indexed vertex of $U_1$ leading to $y$} (see Figure~\ref{fig56a}).
First consider $x=u_q$.  Let one cycle be $R\cup P[z,x]$.  Since $|S|\ge2q-6$
and this cycle contains $S\cup\{x,y,z\}$, it has length at least $2q-3$, which
exceeds $q$.  Meanwhile, $P$ has at least $2q-4$ vertices between $w$ and
$u_{q-1}$.  Since $2q-4\ge q-2$, the edge joining them is oriented as
$wu_{q-1}$.  Hence $Q\cup wu_{q-1}\cup P[u_{q-1},u]$ is a cycle with at least
$q$ vertices.

When $x\ne u_q$, the edge joining $u_{2q-2}$ and $x$ has the desired
orientation, because along $P$ it skips $u_q$ and $q-3$ vertices of $S$.
Hence $R\cup P[z,u_{2q-2}]\cup u_{2q-2}x$ is a cycle.  Since $|S|\ge2q-6$,
it has at least $q-3$ vertices in $S$ plus $\{x,y,z\}$.
The other cycle is $Q\cup wu_{2q-3}\cup P[u_{2q-3},u_{q+1}]\cup u_{q+1}u$.
Since $w$ is earlier than $z$ along $P$, there are at least $q-2$ vertices
between $w$ and $u_{2q-3}$ along $P$; the same is true of $u_{q+1}$ and $u$.
Hence this is a cycle, and $\{\VEC u{2q-3}{q+1}\}\cup\{u,v,w\}$ has at least
$q$ vertices.

\begin{figure}[hbt]
\begin{center}
\gpic{
\expandafter\ifx\csname graph\endcsname\relax \csname newbox\endcsname\graph\fi
\expandafter\ifx\csname graphtemp\endcsname\relax \csname newdimen\endcsname\graphtemp\fi
\setbox\graph=\vtop{\vskip 0pt\hbox{%
    \graphtemp=.5ex\advance\graphtemp by 0.756in
    \rlap{\kern 0.650in\lower\graphtemp\hbox to 0pt{\hss $\bu$\hss}}%
    \graphtemp=.5ex\advance\graphtemp by 0.756in
    \rlap{\kern 0.933in\lower\graphtemp\hbox to 0pt{\hss $\bu$\hss}}%
    \graphtemp=.5ex\advance\graphtemp by 0.756in
    \rlap{\kern 0.933in\lower\graphtemp\hbox to 0pt{\hss $\bu$\hss}}%
    \graphtemp=.5ex\advance\graphtemp by 0.756in
    \rlap{\kern 0.933in\lower\graphtemp\hbox to 0pt{\hss $\bu$\hss}}%
    \graphtemp=.5ex\advance\graphtemp by 0.756in
    \rlap{\kern 2.063in\lower\graphtemp\hbox to 0pt{\hss $\bu$\hss}}%
    \graphtemp=.5ex\advance\graphtemp by 0.756in
    \rlap{\kern 2.346in\lower\graphtemp\hbox to 0pt{\hss $\bu$\hss}}%
    \graphtemp=.5ex\advance\graphtemp by 0.756in
    \rlap{\kern 2.911in\lower\graphtemp\hbox to 0pt{\hss $\bu$\hss}}%
    \graphtemp=.5ex\advance\graphtemp by 0.191in
    \rlap{\kern 0.650in\lower\graphtemp\hbox to 0pt{\hss $\bu$\hss}}%
    \graphtemp=.5ex\advance\graphtemp by 0.191in
    \rlap{\kern 0.933in\lower\graphtemp\hbox to 0pt{\hss $\bu$\hss}}%
    \special{pn 28}%
    \special{pa 2063 756}%
    \special{pa 933 191}%
    \special{fp}%
    \special{pn 8}%
    \special{pa 1611 530}%
    \special{pa 1385 417}%
    \special{fp}%
    \special{sh 1.000}%
    \special{pa 1473 493}%
    \special{pa 1385 417}%
    \special{pa 1499 442}%
    \special{pa 1473 493}%
    \special{fp}%
    \special{pn 28}%
    \special{pa 650 191}%
    \special{pa 650 756}%
    \special{fp}%
    \special{pn 8}%
    \special{pa 650 417}%
    \special{pa 650 530}%
    \special{fp}%
    \special{sh 1.000}%
    \special{pa 678 417}%
    \special{pa 650 530}%
    \special{pa 622 417}%
    \special{pa 678 417}%
    \special{fp}%
    \special{pn 28}%
    \special{pa 2911 756}%
    \special{pa 2346 191}%
    \special{fp}%
    \special{pn 8}%
    \special{pa 2685 530}%
    \special{pa 2572 417}%
    \special{fp}%
    \special{sh 1.000}%
    \special{pa 2632 517}%
    \special{pa 2572 417}%
    \special{pa 2672 477}%
    \special{pa 2632 517}%
    \special{fp}%
    \special{pn 28}%
    \special{pa 933 191}%
    \special{pa 933 756}%
    \special{fp}%
    \special{pn 8}%
    \special{pa 933 417}%
    \special{pa 933 530}%
    \special{fp}%
    \special{sh 1.000}%
    \special{pa 961 417}%
    \special{pa 933 530}%
    \special{pa 904 417}%
    \special{pa 961 417}%
    \special{fp}%
    \special{pn 28}%
    \special{pa 933 756}%
    \special{pa 933 756}%
    \special{fp}%
    \special{pa 933 756}%
    \special{pa 2063 756}%
    \special{fp}%
    \special{pa 2346 756}%
    \special{pa 2911 756}%
    \special{fp}%
    \special{pa 933 756}%
    \special{pa 2063 756}%
    \special{fp}%
    \special{pn 8}%
    \special{pa 650 191}%
    \special{pa 1667 191}%
    \special{fp}%
    \special{sh 1.000}%
    \special{pa 1554 163}%
    \special{pa 1667 191}%
    \special{pa 1554 219}%
    \special{pa 1554 163}%
    \special{fp}%
    \special{pa 1667 191}%
    \special{pa 2346 191}%
    \special{fp}%
    \special{pn 11}%
    \special{pa 650 756}%
    \special{pa 3052 756}%
    \special{fp}%
    \special{pn 8}%
    \special{pa 141 756}%
    \special{pa 447 756}%
    \special{fp}%
    \special{sh 1.000}%
    \special{pa 333 728}%
    \special{pa 447 756}%
    \special{pa 333 784}%
    \special{pa 333 728}%
    \special{fp}%
    \special{pa 447 756}%
    \special{pa 650 756}%
    \special{fp}%
    \special{pn 28}%
    \special{ar 1498 -546 1554 1554 0.993865 2.147728}%
    \special{ar 1498 1978 1978 1978 -2.013707 -1.127885}%
    \graphtemp=.5ex\advance\graphtemp by 0.676in
    \rlap{\kern 0.570in\lower\graphtemp\hbox to 0pt{\hss $w$\hss}}%
    \graphtemp=.5ex\advance\graphtemp by 0.676in
    \rlap{\kern 0.853in\lower\graphtemp\hbox to 0pt{\hss $z$\hss}}%
    \graphtemp=.5ex\advance\graphtemp by 0.643in
    \rlap{\kern 0.791in\lower\graphtemp\hbox to 0pt{\hss $~$\hss}}%
    \graphtemp=.5ex\advance\graphtemp by 0.643in
    \rlap{\kern 0.961in\lower\graphtemp\hbox to 0pt{\hss $~$\hss}}%
    \graphtemp=.5ex\advance\graphtemp by 0.643in
    \rlap{\kern 2.063in\lower\graphtemp\hbox to 0pt{\hss $x$\hss}}%
    \graphtemp=.5ex\advance\graphtemp by 0.643in
    \rlap{\kern 2.346in\lower\graphtemp\hbox to 0pt{\hss $u_{q-1}$\hss}}%
    \graphtemp=.5ex\advance\graphtemp by 0.676in
    \rlap{\kern 2.963in\lower\graphtemp\hbox to 0pt{\hss $u$\hss}}%
    \graphtemp=.5ex\advance\graphtemp by 0.271in
    \rlap{\kern 0.570in\lower\graphtemp\hbox to 0pt{\hss $v$\hss}}%
    \graphtemp=.5ex\advance\graphtemp by 0.271in
    \rlap{\kern 0.853in\lower\graphtemp\hbox to 0pt{\hss $y$\hss}}%
    \special{pn 8}%
    \special{pa 1978 841}%
    \special{pa 3109 841}%
    \special{pa 3109 558}%
    \special{pa 1978 558}%
    \special{pa 1978 841}%
    \special{da 0.057}%
    \special{pa 1074 841}%
    \special{pa 1922 841}%
    \special{pa 1922 558}%
    \special{pa 1074 558}%
    \special{pa 1074 841}%
    \special{da 0.057}%
    \special{pa 113 841}%
    \special{pa 1017 841}%
    \special{pa 1017 558}%
    \special{pa 113 558}%
    \special{pa 113 841}%
    \special{da 0.057}%
    \graphtemp=.5ex\advance\graphtemp by 0.954in
    \rlap{\kern 2.543in\lower\graphtemp\hbox to 0pt{\hss $U_1$\hss}}%
    \graphtemp=.5ex\advance\graphtemp by 0.954in
    \rlap{\kern 1.498in\lower\graphtemp\hbox to 0pt{\hss $S$\hss}}%
    \graphtemp=.5ex\advance\graphtemp by 0.954in
    \rlap{\kern 0.565in\lower\graphtemp\hbox to 0pt{\hss $U_2$\hss}}%
    \graphtemp=.5ex\advance\graphtemp by 0.134in
    \rlap{\kern 0.424in\lower\graphtemp\hbox to 0pt{\hss $C^*$\hss}}%
    \graphtemp=.5ex\advance\graphtemp by 0.700in
    \rlap{\kern 0.000in\lower\graphtemp\hbox to 0pt{\hss $P$\hss}}%
    \graphtemp=.5ex\advance\graphtemp by 0.756in
    \rlap{\kern 4.041in\lower\graphtemp\hbox to 0pt{\hss $\bu$\hss}}%
    \graphtemp=.5ex\advance\graphtemp by 0.756in
    \rlap{\kern 4.324in\lower\graphtemp\hbox to 0pt{\hss $\bu$\hss}}%
    \graphtemp=.5ex\advance\graphtemp by 0.756in
    \rlap{\kern 4.804in\lower\graphtemp\hbox to 0pt{\hss $\bu$\hss}}%
    \graphtemp=.5ex\advance\graphtemp by 0.756in
    \rlap{\kern 5.087in\lower\graphtemp\hbox to 0pt{\hss $\bu$\hss}}%
    \graphtemp=.5ex\advance\graphtemp by 0.756in
    \rlap{\kern 5.454in\lower\graphtemp\hbox to 0pt{\hss $\bu$\hss}}%
    \graphtemp=.5ex\advance\graphtemp by 0.756in
    \rlap{\kern 6.020in\lower\graphtemp\hbox to 0pt{\hss $\bu$\hss}}%
    \graphtemp=.5ex\advance\graphtemp by 0.756in
    \rlap{\kern 6.302in\lower\graphtemp\hbox to 0pt{\hss $\bu$\hss}}%
    \graphtemp=.5ex\advance\graphtemp by 0.191in
    \rlap{\kern 4.041in\lower\graphtemp\hbox to 0pt{\hss $\bu$\hss}}%
    \graphtemp=.5ex\advance\graphtemp by 0.191in
    \rlap{\kern 4.324in\lower\graphtemp\hbox to 0pt{\hss $\bu$\hss}}%
    \special{pn 28}%
    \special{pa 6020 756}%
    \special{pa 4324 191}%
    \special{fp}%
    \special{pn 8}%
    \special{pa 5341 530}%
    \special{pa 5002 417}%
    \special{fp}%
    \special{sh 1.000}%
    \special{pa 5100 480}%
    \special{pa 5002 417}%
    \special{pa 5118 426}%
    \special{pa 5100 480}%
    \special{fp}%
    \special{pn 28}%
    \special{pa 4041 191}%
    \special{pa 4041 756}%
    \special{fp}%
    \special{pn 8}%
    \special{pa 4041 417}%
    \special{pa 4041 530}%
    \special{fp}%
    \special{sh 1.000}%
    \special{pa 4070 417}%
    \special{pa 4041 530}%
    \special{pa 4013 417}%
    \special{pa 4070 417}%
    \special{fp}%
    \special{pn 28}%
    \special{pa 6302 756}%
    \special{pa 5737 191}%
    \special{fp}%
    \special{pn 8}%
    \special{pa 6076 530}%
    \special{pa 5963 417}%
    \special{fp}%
    \special{sh 1.000}%
    \special{pa 6023 517}%
    \special{pa 5963 417}%
    \special{pa 6063 477}%
    \special{pa 6023 517}%
    \special{fp}%
    \special{pn 28}%
    \special{pa 4324 191}%
    \special{pa 4324 756}%
    \special{fp}%
    \special{pn 8}%
    \special{pa 4324 417}%
    \special{pa 4324 530}%
    \special{fp}%
    \special{sh 1.000}%
    \special{pa 4352 417}%
    \special{pa 4324 530}%
    \special{pa 4296 417}%
    \special{pa 4352 417}%
    \special{fp}%
    \special{pn 28}%
    \special{pa 4324 756}%
    \special{pa 4804 756}%
    \special{fp}%
    \special{pa 5087 756}%
    \special{pa 5454 756}%
    \special{fp}%
    \special{pn 8}%
    \special{pa 4041 191}%
    \special{pa 5059 191}%
    \special{fp}%
    \special{sh 1.000}%
    \special{pa 4946 163}%
    \special{pa 5059 191}%
    \special{pa 4946 219}%
    \special{pa 4946 163}%
    \special{fp}%
    \special{pa 5059 191}%
    \special{pa 5737 191}%
    \special{fp}%
    \special{pn 11}%
    \special{pa 4041 756}%
    \special{pa 6443 756}%
    \special{fp}%
    \special{pn 8}%
    \special{pa 3533 756}%
    \special{pa 3838 756}%
    \special{fp}%
    \special{sh 1.000}%
    \special{pa 3725 728}%
    \special{pa 3838 756}%
    \special{pa 3725 784}%
    \special{pa 3725 728}%
    \special{fp}%
    \special{pa 3838 756}%
    \special{pa 4041 756}%
    \special{fp}%
    \special{pn 28}%
    \special{ar 4564 -116 1017 1017 1.031084 2.110508}%
    \special{ar 5412 -487 1385 1385 1.116561 2.025032}%
    \special{ar 5878 122 763 763 0.981765 2.159827}%
    \special{ar 4889 1978 1978 1978 -2.013707 -1.127885}%
    \graphtemp=.5ex\advance\graphtemp by 0.676in
    \rlap{\kern 3.961in\lower\graphtemp\hbox to 0pt{\hss $w$\hss}}%
    \graphtemp=.5ex\advance\graphtemp by 0.676in
    \rlap{\kern 4.244in\lower\graphtemp\hbox to 0pt{\hss $z$\hss}}%
    \graphtemp=.5ex\advance\graphtemp by 0.643in
    \rlap{\kern 4.663in\lower\graphtemp\hbox to 0pt{\hss $u_{2q-2}$\hss}}%
    \graphtemp=.5ex\advance\graphtemp by 0.643in
    \rlap{\kern 5.087in\lower\graphtemp\hbox to 0pt{\hss $u_{2q-3}$\hss}}%
    \graphtemp=.5ex\advance\graphtemp by 0.643in
    \rlap{\kern 5.454in\lower\graphtemp\hbox to 0pt{\hss $u_{q+1}$\hss}}%
    \graphtemp=.5ex\advance\graphtemp by 0.643in
    \rlap{\kern 6.020in\lower\graphtemp\hbox to 0pt{\hss $x$\hss}}%
    \graphtemp=.5ex\advance\graphtemp by 0.676in
    \rlap{\kern 6.354in\lower\graphtemp\hbox to 0pt{\hss $u$\hss}}%
    \graphtemp=.5ex\advance\graphtemp by 0.271in
    \rlap{\kern 3.961in\lower\graphtemp\hbox to 0pt{\hss $v$\hss}}%
    \graphtemp=.5ex\advance\graphtemp by 0.271in
    \rlap{\kern 4.244in\lower\graphtemp\hbox to 0pt{\hss $y$\hss}}%
    \special{pn 8}%
    \special{pa 5652 841}%
    \special{pa 6500 841}%
    \special{pa 6500 558}%
    \special{pa 5652 558}%
    \special{pa 5652 841}%
    \special{da 0.057}%
    \special{pa 4465 841}%
    \special{pa 5596 841}%
    \special{pa 5596 558}%
    \special{pa 4465 558}%
    \special{pa 4465 841}%
    \special{da 0.057}%
    \special{pa 3504 841}%
    \special{pa 4409 841}%
    \special{pa 4409 558}%
    \special{pa 3504 558}%
    \special{pa 3504 841}%
    \special{da 0.057}%
    \graphtemp=.5ex\advance\graphtemp by 0.954in
    \rlap{\kern 6.161in\lower\graphtemp\hbox to 0pt{\hss $U_1$\hss}}%
    \graphtemp=.5ex\advance\graphtemp by 0.954in
    \rlap{\kern 4.974in\lower\graphtemp\hbox to 0pt{\hss $S$\hss}}%
    \graphtemp=.5ex\advance\graphtemp by 0.954in
    \rlap{\kern 3.957in\lower\graphtemp\hbox to 0pt{\hss $U_2$\hss}}%
    \graphtemp=.5ex\advance\graphtemp by 0.134in
    \rlap{\kern 3.815in\lower\graphtemp\hbox to 0pt{\hss $C^*$\hss}}%
    \graphtemp=.5ex\advance\graphtemp by 0.700in
    \rlap{\kern 3.391in\lower\graphtemp\hbox to 0pt{\hss $P$\hss}}%
    \hbox{\vrule depth1.008in width0pt height 0pt}%
    \kern 6.500in
  }%
}%
}
\vspace{-1pc}
\caption{Case 1 of Lemma~\ref{ClaimD}.\label{fig56a} }
\end{center}
\end{figure}
\vspace{-1pc}

\smallskip
{\bf Case 2:} {\it $z$ is farther from $S$ than $w$ along $P$, so $x$ is the
lowest indexed vertex of $U_1$ leading to $y$.}
Let $x=u_t$, and define the paths $Q$ and $R$ as in Case 1.  Note that $t\le q$.
We want the two cycles to be $R\cup zu_{t+q-3}\cup P[u_{t+q-3},x]$
and [$Q\cup P[w,u_{t+q-2}]\cup u_{t+q-2}u_{t-1}\cup P[u_{t-1},u]$.
The jumps along $P$ must skip at least $q-2$ vertices.  This is explicit
for $u_{t+q-2}u_{t-1}$.  Since $(z,w)=(u_j,u_i)$ with $j>i>4q-11$,
and $j-(t+q-3)\ge 3q-6-t\ge 2q-6$, the other construction is also a cycle
if $2q-6\ge q-1$, which holds when $q\ge5$.

For length at least $q$, the first cycle has $q-3$ vertices along $P$ plus at
least $\{x,y,z\}$, and the second adds to $\{u,v,w\}$ all of $S\cup U_1$ except
the $q-2$ vertices used by the first cycle and $u_1$ and maybe $u_2$.  Since
$4q-11-q\ge q-3$ when $q\ge4$, both cycles are long enough.
\end{proof}

\vspace{-1pc}
\begin{figure}[hbt]
\begin{center}
\gpic{
\expandafter\ifx\csname graph\endcsname\relax \csname newbox\endcsname\graph\fi
\expandafter\ifx\csname graphtemp\endcsname\relax \csname newdimen\endcsname\graphtemp\fi
\setbox\graph=\vtop{\vskip 0pt\hbox{%
    \graphtemp=.5ex\advance\graphtemp by 0.865in
    \rlap{\kern 0.809in\lower\graphtemp\hbox to 0pt{\hss $\bu$\hss}}%
    \graphtemp=.5ex\advance\graphtemp by 0.865in
    \rlap{\kern 1.132in\lower\graphtemp\hbox to 0pt{\hss $\bu$\hss}}%
    \graphtemp=.5ex\advance\graphtemp by 0.865in
    \rlap{\kern 1.779in\lower\graphtemp\hbox to 0pt{\hss $\bu$\hss}}%
    \graphtemp=.5ex\advance\graphtemp by 0.865in
    \rlap{\kern 2.102in\lower\graphtemp\hbox to 0pt{\hss $\bu$\hss}}%
    \graphtemp=.5ex\advance\graphtemp by 0.865in
    \rlap{\kern 2.749in\lower\graphtemp\hbox to 0pt{\hss $\bu$\hss}}%
    \graphtemp=.5ex\advance\graphtemp by 0.865in
    \rlap{\kern 3.072in\lower\graphtemp\hbox to 0pt{\hss $\bu$\hss}}%
    \graphtemp=.5ex\advance\graphtemp by 0.865in
    \rlap{\kern 3.396in\lower\graphtemp\hbox to 0pt{\hss $\bu$\hss}}%
    \graphtemp=.5ex\advance\graphtemp by 0.218in
    \rlap{\kern 0.809in\lower\graphtemp\hbox to 0pt{\hss $\bu$\hss}}%
    \graphtemp=.5ex\advance\graphtemp by 0.218in
    \rlap{\kern 1.132in\lower\graphtemp\hbox to 0pt{\hss $\bu$\hss}}%
    \special{pn 28}%
    \special{pa 2749 865}%
    \special{pa 1132 218}%
    \special{fp}%
    \special{pn 8}%
    \special{pa 2102 607}%
    \special{pa 1779 477}%
    \special{fp}%
    \special{sh 1.000}%
    \special{pa 1887 555}%
    \special{pa 1779 477}%
    \special{pa 1911 495}%
    \special{pa 1887 555}%
    \special{fp}%
    \special{pn 28}%
    \special{pa 970 542}%
    \special{pa 809 865}%
    \special{fp}%
    \special{pn 8}%
    \special{pa 906 671}%
    \special{pa 873 736}%
    \special{fp}%
    \special{sh 1.000}%
    \special{pa 960 635}%
    \special{pa 873 736}%
    \special{pa 902 606}%
    \special{pa 960 635}%
    \special{fp}%
    \special{pn 28}%
    \special{pa 3396 865}%
    \special{pa 2749 218}%
    \special{fp}%
    \special{pn 8}%
    \special{pa 3137 607}%
    \special{pa 3008 477}%
    \special{fp}%
    \special{sh 1.000}%
    \special{pa 3076 592}%
    \special{pa 3008 477}%
    \special{pa 3122 546}%
    \special{pa 3076 592}%
    \special{fp}%
    \special{pn 28}%
    \special{pa 970 542}%
    \special{pa 1132 865}%
    \special{fp}%
    \special{pn 8}%
    \special{pa 1035 671}%
    \special{pa 1067 736}%
    \special{fp}%
    \special{sh 1.000}%
    \special{pa 1038 606}%
    \special{pa 1067 736}%
    \special{pa 980 635}%
    \special{pa 1038 606}%
    \special{fp}%
    \special{pn 28}%
    \special{pa 1132 865}%
    \special{pa 1779 865}%
    \special{fp}%
    \special{pa 2102 865}%
    \special{pa 2749 865}%
    \special{fp}%
    \special{pa 3072 865}%
    \special{pa 3396 865}%
    \special{fp}%
    \special{pn 8}%
    \special{pa 809 218}%
    \special{pa 1973 218}%
    \special{fp}%
    \special{sh 1.000}%
    \special{pa 1843 186}%
    \special{pa 1973 218}%
    \special{pa 1843 251}%
    \special{pa 1843 186}%
    \special{fp}%
    \special{pa 1973 218}%
    \special{pa 2749 218}%
    \special{fp}%
    \special{pn 28}%
    \special{pa 1132 218}%
    \special{pa 970 542}%
    \special{fp}%
    \special{pa 809 218}%
    \special{pa 970 542}%
    \special{fp}%
    \special{pn 11}%
    \special{pa 809 865}%
    \special{pa 3719 865}%
    \special{fp}%
    \special{pn 8}%
    \special{pa 162 865}%
    \special{pa 550 865}%
    \special{fp}%
    \special{sh 1.000}%
    \special{pa 420 833}%
    \special{pa 550 865}%
    \special{pa 420 898}%
    \special{pa 420 833}%
    \special{fp}%
    \special{pa 550 865}%
    \special{pa 809 865}%
    \special{fp}%
    \special{pn 28}%
    \special{ar 1779 2264 2264 2264 -2.013707 -1.127885}%
    \special{ar 1455 -254 1294 1294 1.047198 2.094395}%
    \special{ar 2426 -254 1294 1294 1.047198 2.094395}%
    \graphtemp=.5ex\advance\graphtemp by 0.774in
    \rlap{\kern 0.749in\lower\graphtemp\hbox to 0pt{\hss $z$\hss}}%
    \graphtemp=.5ex\advance\graphtemp by 0.774in
    \rlap{\kern 1.191in\lower\graphtemp\hbox to 0pt{\hss $w$\hss}}%
    \graphtemp=.5ex\advance\graphtemp by 0.736in
    \rlap{\kern 1.617in\lower\graphtemp\hbox to 0pt{\hss $u_{t+q-2}$\hss}}%
    \graphtemp=.5ex\advance\graphtemp by 0.736in
    \rlap{\kern 2.134in\lower\graphtemp\hbox to 0pt{\hss $u_{t+q-3}$\hss}}%
    \graphtemp=.5ex\advance\graphtemp by 0.736in
    \rlap{\kern 2.749in\lower\graphtemp\hbox to 0pt{\hss $x$\hss}}%
    \graphtemp=.5ex\advance\graphtemp by 0.736in
    \rlap{\kern 3.072in\lower\graphtemp\hbox to 0pt{\hss $u_{t-1}$\hss}}%
    \graphtemp=.5ex\advance\graphtemp by 0.774in
    \rlap{\kern 3.455in\lower\graphtemp\hbox to 0pt{\hss $u$\hss}}%
    \graphtemp=.5ex\advance\graphtemp by 0.310in
    \rlap{\kern 0.749in\lower\graphtemp\hbox to 0pt{\hss $v$\hss}}%
    \graphtemp=.5ex\advance\graphtemp by 0.310in
    \rlap{\kern 1.159in\lower\graphtemp\hbox to 0pt{\hss $y$\hss}}%
    \special{pn 8}%
    \special{pa 2474 962}%
    \special{pa 3800 962}%
    \special{pa 3800 639}%
    \special{pa 2474 639}%
    \special{pa 2474 962}%
    \special{da 0.065}%
    \special{pa 1358 962}%
    \special{pa 2393 962}%
    \special{pa 2393 639}%
    \special{pa 1358 639}%
    \special{pa 1358 962}%
    \special{da 0.065}%
    \special{pa 129 962}%
    \special{pa 1294 962}%
    \special{pa 1294 639}%
    \special{pa 129 639}%
    \special{pa 129 962}%
    \special{da 0.065}%
    \graphtemp=.5ex\advance\graphtemp by 1.092in
    \rlap{\kern 3.137in\lower\graphtemp\hbox to 0pt{\hss $U_1$\hss}}%
    \graphtemp=.5ex\advance\graphtemp by 1.092in
    \rlap{\kern 1.876in\lower\graphtemp\hbox to 0pt{\hss $S$\hss}}%
    \graphtemp=.5ex\advance\graphtemp by 1.092in
    \rlap{\kern 0.711in\lower\graphtemp\hbox to 0pt{\hss $U_2$\hss}}%
    \graphtemp=.5ex\advance\graphtemp by 0.154in
    \rlap{\kern 0.550in\lower\graphtemp\hbox to 0pt{\hss $C^*$\hss}}%
    \graphtemp=.5ex\advance\graphtemp by 0.801in
    \rlap{\kern 0.000in\lower\graphtemp\hbox to 0pt{\hss $P$\hss}}%
    \hbox{\vrule depth1.092in width0pt height 0pt}%
    \kern 3.800in
  }%
}%
}
\vspace{-1pc}
\caption{Case 2 of Lemma~\ref{ClaimD}.\label{fig56b} }
\end{center}
\end{figure}
\vspace{-1pc}

\begin{lemma}\label{ClaimE}
Let $C$ and $C'$ be two members of ${\cF}$, with
$W=V(C)$ and $W'=V(C')$.  If $T$ contains a $q$-matching from $U_{1}$ to
$W$ and a 3-matching from $W'$ to $U_2$, and $d^{+}(W,W'))\ge q(q-1)+3$,
then there is an extension of ${\cF}$.
\end{lemma}

\begin{proof}
Again use the same notation.  We may assume that $T[V(P)]$ contains no
$q$-cycle.  We will extend ${\cF}$ by replacing $C$ and $C'$ in ${\cF}$ with
three $q$-cycles (except in one case).  We must ensure that they introduce at
most $3q-6$ new vertices.  The other members of ${\cF}$ remain.

Since $d^+(W,W')\ge q(q-1)+3$, the set $W$ has at least three vertices that
each dominate $W'$ (otherwise, $d^+(W,W')\le 2q+(q-2)(q-1)=q(q-1)+2$).  Label
these as $x_1,x_2,x_3$ so that with $w_1x_1,w_2x_2,w_3x_3$ being edges in the
$q$-matching from $U_1$ to $W$, the vertices $w_1,w_2,w_3$ occur in that order
along $P$ through $U_1$ ($w_1$ is closest to $S$).  See Figure~\ref{fig57}.

Let the edges of the given $3$-matching from $W'$ to $U_2$ be
$y_1z_1,y_2z_2,y_3z_3$, indexed so that $z_3,z_1,z_2$ occur in that order along
$P$ ($z_2$ is closest to $S$).  Since each vertex in $\{x_1,x_2,x_3\}$
dominates $W'$, we now have three disjoint paths from $U_1$ to $U_2$; the
$i$th path $R_i$ is $\la w_i,x_i,y_i,z_i\ra$.

We complete these three paths to disjoint cycles by adding vertices along the
path $P$.  Recall that $P$ contains the path $\la \VEC u{4q-11}{q+1}\ra$
through $S$ between $U_2$ and $U_1$.  Along this path of $3q-11$ vertices
define three disjoint paths, each having $q-4$ vertices (one vertex of $S$ is
not needed); call them $Q_2,Q_3,Q_1$ in order along $P$.

Let $B_i$ be the cycle formed by combining $R_i$ and $Q_i$; add the edges from
the end of each of $R_i$ and $Q_i$ to the beginning of the other, except that
between $z_2$ and $Q_2$ in $B_2$ we follow $P$ to the end of $U_2$, and between
$Q_1$ and $w_1$ in $B_1$ we follow $P$ through the beginning of $U_1$.  The
edges from $z_1$ to $Q_1$, from $z_3$ to $Q_3$, from $Q_3$ to $w_3$, and from
$Q_2$ to $w_2$, are oriented in the desired direction because they skip at
least $q-2$ vertices along $P$ and hence would complete cycles of length at
least $q$ if oriented in the other direction.

Note that $B_3$ has exactly $q$ vertices.  Possibly $B_1$ or $B_2$ has more
vertices due to picking up extras at the beginning of $U_1$ or the end of $U_2$.
However, we can shorten the cycles to length $q$ by omitting vertices at the
beginning of $Q_1$ and/or the end of $Q_2$.  This only makes the jumps along
$P$ longer, so the edges make the jumps still have the desired orientation.
The resulting cycle $B'_i$ is a $q$-cycle using exactly two vertices used in
$\cF$ ($x_i$ and $y_i$).  Hence $\{B'_1,B'_2,B_3\}$ replaces $\{C,C'\}$ to
yield an extension using $3q-6$ vertices not used by $\cF$.
\end{proof}


\vspace{-1pc}
\begin{figure}[hbt]
\begin{center}
\gpic{
\expandafter\ifx\csname graph\endcsname\relax \csname newbox\endcsname\graph\fi
\expandafter\ifx\csname graphtemp\endcsname\relax \csname newdimen\endcsname\graphtemp\fi
\setbox\graph=\vtop{\vskip 0pt\hbox{%
    \special{pn 11}%
    \special{pa 295 1176}%
    \special{pa 5402 1176}%
    \special{fp}%
    \graphtemp=.5ex\advance\graphtemp by 1.176in
    \rlap{\kern 0.688in\lower\graphtemp\hbox to 0pt{\hss $\bu$\hss}}%
    \graphtemp=.5ex\advance\graphtemp by 1.176in
    \rlap{\kern 1.080in\lower\graphtemp\hbox to 0pt{\hss $\bu$\hss}}%
    \graphtemp=.5ex\advance\graphtemp by 1.176in
    \rlap{\kern 1.473in\lower\graphtemp\hbox to 0pt{\hss $\bu$\hss}}%
    \graphtemp=.5ex\advance\graphtemp by 0.272in
    \rlap{\kern 1.080in\lower\graphtemp\hbox to 0pt{\hss $\bu$\hss}}%
    \graphtemp=.5ex\advance\graphtemp by 0.390in
    \rlap{\kern 1.473in\lower\graphtemp\hbox to 0pt{\hss $\bu$\hss}}%
    \graphtemp=.5ex\advance\graphtemp by 0.508in
    \rlap{\kern 1.866in\lower\graphtemp\hbox to 0pt{\hss $\bu$\hss}}%
    \special{pn 28}%
    \special{pa 1080 272}%
    \special{pa 1080 783}%
    \special{fp}%
    \special{pn 8}%
    \special{pa 1080 476}%
    \special{pa 1080 578}%
    \special{fp}%
    \special{sh 1.000}%
    \special{pa 1120 421}%
    \special{pa 1080 578}%
    \special{pa 1041 421}%
    \special{pa 1120 421}%
    \special{fp}%
    \special{pn 28}%
    \special{pa 1866 508}%
    \special{pa 1473 1176}%
    \special{fp}%
    \special{pn 8}%
    \special{pa 1709 775}%
    \special{pa 1630 908}%
    \special{fp}%
    \special{sh 1.000}%
    \special{pa 1744 793}%
    \special{pa 1630 908}%
    \special{pa 1676 753}%
    \special{pa 1744 793}%
    \special{fp}%
    \special{pn 28}%
    \special{pa 1473 390}%
    \special{pa 688 1176}%
    \special{fp}%
    \special{pn 8}%
    \special{pa 1159 704}%
    \special{pa 1002 861}%
    \special{fp}%
    \special{sh 1.000}%
    \special{pa 1141 778}%
    \special{pa 1002 861}%
    \special{pa 1085 722}%
    \special{pa 1141 778}%
    \special{fp}%
    \special{pn 11}%
    \special{pa 1080 272}%
    \special{pa 1866 508}%
    \special{fp}%
    \special{pn 8}%
    \special{ar 1666 -251 786 786 1.313041 2.411465}%
    \special{pn 28}%
    \special{pa 1080 783}%
    \special{pa 1080 1176}%
    \special{fp}%
    \graphtemp=.5ex\advance\graphtemp by 1.293in
    \rlap{\kern 0.688in\lower\graphtemp\hbox to 0pt{\hss $z_3$\hss}}%
    \graphtemp=.5ex\advance\graphtemp by 1.293in
    \rlap{\kern 1.080in\lower\graphtemp\hbox to 0pt{\hss $z_1$\hss}}%
    \graphtemp=.5ex\advance\graphtemp by 1.293in
    \rlap{\kern 1.473in\lower\graphtemp\hbox to 0pt{\hss $z_2$\hss}}%
    \graphtemp=.5ex\advance\graphtemp by 0.154in
    \rlap{\kern 1.080in\lower\graphtemp\hbox to 0pt{\hss $y_1$\hss}}%
    \graphtemp=.5ex\advance\graphtemp by 0.272in
    \rlap{\kern 1.473in\lower\graphtemp\hbox to 0pt{\hss $y_3$\hss}}%
    \graphtemp=.5ex\advance\graphtemp by 0.390in
    \rlap{\kern 1.866in\lower\graphtemp\hbox to 0pt{\hss $y_2$\hss}}%
    \graphtemp=.5ex\advance\graphtemp by 1.176in
    \rlap{\kern 4.223in\lower\graphtemp\hbox to 0pt{\hss $\bu$\hss}}%
    \graphtemp=.5ex\advance\graphtemp by 1.176in
    \rlap{\kern 4.616in\lower\graphtemp\hbox to 0pt{\hss $\bu$\hss}}%
    \graphtemp=.5ex\advance\graphtemp by 1.176in
    \rlap{\kern 5.009in\lower\graphtemp\hbox to 0pt{\hss $\bu$\hss}}%
    \graphtemp=.5ex\advance\graphtemp by 0.508in
    \rlap{\kern 3.830in\lower\graphtemp\hbox to 0pt{\hss $\bu$\hss}}%
    \graphtemp=.5ex\advance\graphtemp by 0.390in
    \rlap{\kern 4.223in\lower\graphtemp\hbox to 0pt{\hss $\bu$\hss}}%
    \graphtemp=.5ex\advance\graphtemp by 0.272in
    \rlap{\kern 4.616in\lower\graphtemp\hbox to 0pt{\hss $\bu$\hss}}%
    \special{pa 4223 1176}%
    \special{pa 4223 390}%
    \special{fp}%
    \special{pn 8}%
    \special{pa 4223 861}%
    \special{pa 4223 704}%
    \special{fp}%
    \special{sh 1.000}%
    \special{pa 4184 861}%
    \special{pa 4223 704}%
    \special{pa 4263 861}%
    \special{pa 4184 861}%
    \special{fp}%
    \special{pn 28}%
    \special{pa 4616 1176}%
    \special{pa 4616 272}%
    \special{fp}%
    \special{pn 8}%
    \special{pa 4616 814}%
    \special{pa 4616 633}%
    \special{fp}%
    \special{sh 1.000}%
    \special{pa 4577 791}%
    \special{pa 4616 633}%
    \special{pa 4655 791}%
    \special{pa 4577 791}%
    \special{fp}%
    \special{pn 28}%
    \special{pa 5009 1176}%
    \special{pa 3830 508}%
    \special{fp}%
    \special{pn 8}%
    \special{pa 4538 908}%
    \special{pa 4302 775}%
    \special{fp}%
    \special{sh 1.000}%
    \special{pa 4419 887}%
    \special{pa 4302 775}%
    \special{pa 4458 818}%
    \special{pa 4419 887}%
    \special{fp}%
    \special{pn 11}%
    \special{pa 3830 508}%
    \special{pa 4616 272}%
    \special{fp}%
    \special{pn 8}%
    \special{ar 4031 -251 786 786 0.730127 1.828552}%
    \graphtemp=.5ex\advance\graphtemp by 1.293in
    \rlap{\kern 4.184in\lower\graphtemp\hbox to 0pt{\hss $w_1$\hss}}%
    \graphtemp=.5ex\advance\graphtemp by 1.293in
    \rlap{\kern 4.655in\lower\graphtemp\hbox to 0pt{\hss $w_2$\hss}}%
    \graphtemp=.5ex\advance\graphtemp by 1.293in
    \rlap{\kern 5.087in\lower\graphtemp\hbox to 0pt{\hss $w_3$\hss}}%
    \graphtemp=.5ex\advance\graphtemp by 0.390in
    \rlap{\kern 3.830in\lower\graphtemp\hbox to 0pt{\hss $x_3$\hss}}%
    \graphtemp=.5ex\advance\graphtemp by 0.272in
    \rlap{\kern 4.223in\lower\graphtemp\hbox to 0pt{\hss $x_1$\hss}}%
    \graphtemp=.5ex\advance\graphtemp by 0.154in
    \rlap{\kern 4.616in\lower\graphtemp\hbox to 0pt{\hss $x_2$\hss}}%
    \special{pn 28}%
    \special{ar 2517 3929 3929 3929 -1.945138 -1.121490}%
    \special{ar 3519 3633 3536 3536 -2.057275 -1.255328}%
    \special{ar 2506 3358 3143 3143 -1.905740 -1.135936}%
    \graphtemp=.5ex\advance\graphtemp by 1.176in
    \rlap{\kern 1.827in\lower\graphtemp\hbox to 0pt{\hss $\bu$\hss}}%
    \graphtemp=.5ex\advance\graphtemp by 1.176in
    \rlap{\kern 2.377in\lower\graphtemp\hbox to 0pt{\hss $\bu$\hss}}%
    \graphtemp=.5ex\advance\graphtemp by 1.176in
    \rlap{\kern 2.534in\lower\graphtemp\hbox to 0pt{\hss $\bu$\hss}}%
    \graphtemp=.5ex\advance\graphtemp by 1.176in
    \rlap{\kern 3.163in\lower\graphtemp\hbox to 0pt{\hss $\bu$\hss}}%
    \graphtemp=.5ex\advance\graphtemp by 1.176in
    \rlap{\kern 3.320in\lower\graphtemp\hbox to 0pt{\hss $\bu$\hss}}%
    \graphtemp=.5ex\advance\graphtemp by 1.176in
    \rlap{\kern 3.870in\lower\graphtemp\hbox to 0pt{\hss $\bu$\hss}}%
    \special{pa 1827 1176}%
    \special{pa 2377 1176}%
    \special{fp}%
    \special{pn 8}%
    \special{pa 2047 1176}%
    \special{pa 2157 1176}%
    \special{fp}%
    \special{sh 1.000}%
    \special{pa 2000 1136}%
    \special{pa 2157 1176}%
    \special{pa 2000 1215}%
    \special{pa 2000 1136}%
    \special{fp}%
    \special{pn 28}%
    \special{pa 2534 1176}%
    \special{pa 3163 1176}%
    \special{fp}%
    \special{pn 8}%
    \special{pa 2785 1176}%
    \special{pa 2911 1176}%
    \special{fp}%
    \special{sh 1.000}%
    \special{pa 2754 1136}%
    \special{pa 2911 1176}%
    \special{pa 2754 1215}%
    \special{pa 2754 1136}%
    \special{fp}%
    \special{pn 28}%
    \special{pa 3320 1176}%
    \special{pa 3870 1176}%
    \special{fp}%
    \special{pn 8}%
    \special{pa 3540 1176}%
    \special{pa 3650 1176}%
    \special{fp}%
    \special{sh 1.000}%
    \special{pa 3492 1136}%
    \special{pa 3650 1176}%
    \special{pa 3492 1215}%
    \special{pa 3492 1136}%
    \special{fp}%
    \special{pn 28}%
    \special{pa 1473 1176}%
    \special{pa 1827 1176}%
    \special{fp}%
    \special{pa 3870 1176}%
    \special{pa 4223 1176}%
    \special{fp}%
    \special{ar 1611 265 1296 1296 0.778281 2.363311}%
    \special{ar 4086 265 1296 1296 0.778281 2.363311}%
    \special{ar 2200 -191 1768 1768 0.884943 2.256649}%
    \special{ar 3496 -191 1768 1768 0.884943 2.256649}%
    \special{pn 8}%
    \special{pa 4046 1451}%
    \special{pa 5500 1451}%
    \special{pa 5500 1058}%
    \special{pa 4046 1058}%
    \special{pa 4046 1451}%
    \special{da 0.079}%
    \special{pa 1729 1451}%
    \special{pa 3968 1451}%
    \special{pa 3968 1058}%
    \special{pa 1729 1058}%
    \special{pa 1729 1451}%
    \special{da 0.079}%
    \special{pa 157 1451}%
    \special{pa 1611 1451}%
    \special{pa 1611 1058}%
    \special{pa 157 1058}%
    \special{pa 157 1451}%
    \special{da 0.079}%
    \graphtemp=.5ex\advance\graphtemp by 0.901in
    \rlap{\kern 5.166in\lower\graphtemp\hbox to 0pt{\hss $U_1$\hss}}%
    \graphtemp=.5ex\advance\graphtemp by 0.901in
    \rlap{\kern 2.848in\lower\graphtemp\hbox to 0pt{\hss $S$\hss}}%
    \graphtemp=.5ex\advance\graphtemp by 0.901in
    \rlap{\kern 0.491in\lower\graphtemp\hbox to 0pt{\hss $U_2$\hss}}%
    \graphtemp=.5ex\advance\graphtemp by 0.390in
    \rlap{\kern 4.852in\lower\graphtemp\hbox to 0pt{\hss $C$\hss}}%
    \graphtemp=.5ex\advance\graphtemp by 0.390in
    \rlap{\kern 0.845in\lower\graphtemp\hbox to 0pt{\hss $C'$\hss}}%
    \graphtemp=.5ex\advance\graphtemp by 1.254in
    \rlap{\kern 0.000in\lower\graphtemp\hbox to 0pt{\hss $P$\hss}}%
    \graphtemp=.5ex\advance\graphtemp by 1.333in
    \rlap{\kern 2.062in\lower\graphtemp\hbox to 0pt{\hss $Q_2$\hss}}%
    \graphtemp=.5ex\advance\graphtemp by 1.333in
    \rlap{\kern 2.809in\lower\graphtemp\hbox to 0pt{\hss $Q_3$\hss}}%
    \graphtemp=.5ex\advance\graphtemp by 1.333in
    \rlap{\kern 3.634in\lower\graphtemp\hbox to 0pt{\hss $Q_1$\hss}}%
    \hbox{\vrule depth1.575in width0pt height 0pt}%
    \kern 5.500in
  }%
}%
}
\vspace{-1pc}
\caption{Figure for Lemma~\ref{ClaimE}.\label{fig57} }
\end{center}
\end{figure}
\vspace{-1pc}

Finally we are ready to complete the main proof.

\bigskip
\noindent
{\bf Proof of Theorem~\ref{THMB}.}
We will obtain bounds on the sizes of various sets of edges under the
assumption that ${\cF}$ has no extension.  These will lead to a
contradiction.

Let ${\cL}$ consist of those $q$-cycles in ${\cF}$ receiving at
least $q(q-1)+1$ edges from $U_1$ (thereby guaranteeing a $q$-matching from
$U_1$, by Lemma~\ref{ClaimA}).
Let ${\cM}$ consist of those $q$-cycles in ${\cF}$ sending at least
$6q-13$ edges to $U_2$ (thereby guaranteeing a $3$-matching to $U_2$, by
Lemma~\ref{ClaimC}(ii)).
Let ${\cR} = {\cF} - ({\cL} \cup {\cM})$.
Let $l$, $m$, and $r$, respectively, denote the sizes
of ${\cL}, {\cM}$ and ${\cR}$.
By Lemma~\ref{ClaimD}, ${\cL} \cap {\cM} = \emptyset$.
Hence $\{{\cL},{\cM},{\cR}\}$ is a partition of ${\cF}$,
and
\begin{equation}\label{EQ1}
 l + m + r = k - 1.
\end{equation}

Now consider $d^+(U_1,S\cup U_2)$.  If this is nonzero, then let $u_{q+s}$ be
the highest-indexed (earliest) vertex of $P$ having a predecessor in $U_1$, and
let $u_{q+1-t}$ be the lowest-indexed vertex of $P$ having a successor in
$S\cup U_2$.  Since $P$ gives a path from $u_{q+s}$ to $u_{q+1-t}$, the
tournament induced by these $s+t$ vertices is strong and has a spanning cycle.
Hence $s+t< q$.  Also $d^+(U_1,S\cup U_2)\le st$.  With $s+t\le q-1$, we have
$d^+(U_1,S\cup U_2)\le (q-1)^2/4$.

Next we obtain upper and lower bounds on $d^+(U_1,{\cF})$ in order
to obtain an inequality involving $l$ and $m$.  Using the computation above,
\begin{equation}\label{EQx}
d^+(U_1,{\cF})\ge q[(q-1)k-1]-\CH q2-\FR{(q-1)^2}4.
\end{equation}

To avoid obtaining an extension of $\cF$ via Lemma~\ref{ClaimD}, each cycle in
$\cM$ must avoid a $(q-1)$-matching from $U_1$ and hence receives at most
$q(q-2)$ edges from $U_1$.  By definition, each cycle in ${\cL}$ or ${\cR}$
receives at most $q^2$ or $q(q-1)$ edges from $U_{1}$, respectively (the latter
because otherwise it would be in $\cL$).  Thus
$$d^{+}(U_{1}, {\cF}) \leq q^2l + q(q-2)m + q(q-1)r.$$
After dividing by $q$, we obtain the following inequality:
$$
(q-1)k-1-\FR{q-1}2-\FR{q-2}4-\FR 1{4q}\le ql+(q-2)m+(q-1)r.
$$
Since $l+m+r=k-1$, we can rewrite the right side as $(q-1)(k-1)+l-m$.  Thus
\begin{equation}\label{EQ2}
l-m\ge \FR q4 -1-\FR1{4q},
\end{equation}
and we can drop the $-1/(4q)$ term since $l$ and $m$ are integers.

Finally, we will obtain upper and lower bounds on the number of edges leaving
${\cL} \cup {\cR}$ in order to obtain an inequality that cannot be satisfied.
First, since $\delta^{+}(T)\ge(q-1)k-1$,
\begin{equation}\label{lower}
d^+({\cL}\cup {\cR}, \overline{{\cL}\cup {\cR}})\ge
q(l+r)[(q - 1)k -1] - \CH{q(l+r)}2.
\end{equation}
The absence of extensions imposes bounds on the number of edges leaving
${\cL}$.  Since every cycle in $\cL$ receives a $q$-matching from $U_1$,
Lemma~\ref{ClaimC}(i) and Lemma~\ref{ClaimD}(i) imply
$d^+({\cL},U_2)\le l(3q-7)$.  By Lemma~\ref{ClaimE},
$d^+({\cL},{\cM})\le lm[q(q-1)+2]$.  Also, $d^+(\cR,\cM)\le q^2mr$.
Each cycle in ${\cR}$ sends at most $6q-14$ edges to $|U_2|$ (otherwise it
would be placed in $\cM$), so $d^+({\cR},U_2)\le (3q-7)2r$.
Also $d^+({\cL}\cup{\cR},S)\le(3q-11)q(l+r)$, since $|S|=3q-11$.

The lower bound $d^+(U_1,{\cF})\ge q((q-1)k-1)-\CH q2-(q-1)^2/4$ is from
\eqref{EQx}.  Since every cycle in ${\cM}$ has a $3$-matching to $U_2$, by
Lemma~\ref{ClaimD} there is no $(q-1)$-matching from $U_1$ to a cycle in
${\cM}$; hence $d^+(U_1,{\cM})\le mq(q-2)$.  Thus
$$d^+(U_1,{\cL}\cup{\cR})\ge q[(q-1)k-1]-\CH q2-\FR{(q-1)^2}4-q(q-2)m.$$
Let $\alpha=\CH q2+(q-1)^2/4$.  Using \eqref{EQ1}, we conclude
\begin{align*}
d^{+}({\cL} \cup {\cR}, U_{1})&\le q^2(l+r) - q[(q-1)k-1]+q(q-2)m+\alpha\\
&=q^2(l+r)-q^2k+kq+q+q^2m-2qm+\alpha\\
&=q^2(l\!+\!r\!+\!m\!-\!k\!+\!1)-q^2+kq+q-2qm+\alpha=q(l\!+\!r\!+\!2\!-\!m\!-\!q)+\alpha.
\end{align*}

Collecting the bounds proved above (with some rearrangement), we have
\begin{align}\label{upper}
d^+&(\cL\cup\cR,\overline{\cL\cup\cR})\nonumber\\
&=\quad d^+(\cL\cup\cR,\cM)\quad +d^+(\cL\cup\cR,U_2)\quad
+d^+(\cL\cup\cR,S)\quad +d^+(\cL\cup\cR,U_1)\\
&\le
m[q^2(l\!+\!r)-l(q\!-\!2)]+(3q\!-\!7)(l\!+\!2r)+(3q\!-\!11)q(l\!+\!r)+q(l\!+\!r\!+\!2\!-\!m\!-\!q)+\alpha.\nonumber
\end{align}

Combining \eqref{lower} and \eqref{upper} and
collecting the terms involving $q(l+r)$  yields
\begin{align}\label{maineq}
q(l+r)\Big[(q-1)k-1-\FR{q(l+r)-1}2&-mq-3q+7+\FR7q\Big]\nonumber\\
&\le-ml(q-2)+(3q-7)r-q(m+q-2)+\alpha.
\end{align}

Using \eqref{EQ1}, we simplify the last factor on the left:
\begin{align}\label{factor}
(q-1)k-1-&\FR{q(l+r)-1}2-mq-3q+7+\FR7q\nonumber\\
&=(q-1)k-\FR{q(l+r+m+1)}2-\FR{mq}2-\FR{5q}2+\FR{13}2+\FR7q\nonumber\\
&=\left(\FR q2-1\right)(k-5)-\FR{mq}2+\FR32+\FR7q.
\end{align}
On the right side of \eqref{maineq}, we compute
$\alpha-q(q-2)=\FR{-(q-1)(q-3)}4+1$.
On the left side, we replace $l+r$ with $k-1-m$,
and on the right we replace $r$ with $k-1-m-l$.  The inequality is now
\begin{align*}
q(k-1-m)&\Big[\left(\FR q2-1\right)(k-5)-\FR{mq}2+\FR32 +\FR7q\Big]\\
&\le -ml(q-2)+(3q-7)(k-1-m-l)-qm-\FR{(q-1)(q-3)}4+1.
\end{align*}
The coefficients of $l$ in its only appearances are negative.  Hence for given
$q,k,m$, the inequality can hold only if it holds when $l$ takes its smallest
allowed numerical value.  By \eqref{EQ2}, we have $l\ge m-1+(q/4)$, which
yields $l\ge m+1$ when $q\ge5$ since $l\in\NN$.  Setting $l=m+1$, we now have a
quadratic inequality for $m$ in terms of $k$ and $q$:
\begin{align}\label{prequad}
q(k-1-m)&\Big[\left(\FR q2-1\right)(k-5)-\FR{mq}2+\FR32 +\FR7q\Big]\\
&\le -m(m+1)(q-2)+(3q-7)(k-2m-2)-qm-\FR{(q-1)(q-3)}4+1.\nonumber
\end{align}

We first collect terms to write this as a quadratic inequality for $m$:
\begin{equation}\label{quad}
\left(\FR{q^2}2+q-2\right)m^2+\left[(q-q^2)k+3q^2+\FR32q-23\right]m+c\le 0,
\end{equation}
where $c$ depends only on $k$ and $q$.

The inequality $l\ge m+1$ also yields $k-1=l+m+r\ge 2m+1$, and hence
$m\le \FL{k/2}-1$.  We thus want to show that \eqref{quad} cannot be
satisfied when $k\ge q+1\ge6$ and $0\le m\le \FL{k/2}-1$.

In order to obtain the desired contradiction, it suffices to show that the
left side of \eqref{quad} is positive at its lowest allowed point.  Since the
coefficient of the quadratic term is positive, the quadratic polynomial is
minimized where its derivative is $0$.  This occurs when
\begin{equation}\label{deriv}
(q^2+2q-4)m=(q^2-q)k-(3q^2+\FR32q-23).
\end{equation}

The analysis simplifies if the lowest value of the polynomial in \eqref{quad}
among allowed values for $m$ occurs at the highest allowed value,
$\FL{k/2}-1$.  Since the graph of a quadratic polynomial is symmetric around
the minimum, when $k$ is even this holds if the minimizing point is at least
$k/2-3/2$.  When $k$ is odd this also suffices, due to the floor function.

Thus we want the solution for $m$ in \eqref{deriv} to be at least $(k-3)/2$.
This holds unless $(q^2+2q-4)\FR{k-3}2>(q^2-q)k-(3q^2+\FR32q-23)$.  Solving for
$k$ yields
\begin{equation}\label{kbound}
k<\FR{3q^2-3q-34}{q^2-4q+4}.
\end{equation}
The right side of \eqref{kbound} is less than $5$ for all $q$
(since $5(q^2-4q+4)-(3q^2-3q-34)$ has no root).  Hence in the
case $k\ge q+1\ge5$ we have the desired reduction.

Hence it suffices to show that \eqref{quad} or equivalently \eqref{prequad}
cannot hold when $m=\FL{k/2}-1$.
When $k$ is even, we set $m=k/2-1$ in \eqref{prequad}.
This simplifies the expression, since now $k-2m-2=0$ and $k-1-m=k/2$.
We require
\begin{equation}\label{keven}
q\FR k2\left[\left(\FR q2\!-\!1\right)(k\!-\!5)
-\FR{(k\!-\!2)q}4\!+\!\FR32\!+\!\FR7q\right]
\le -\FR{(k\!-\!2)k}4(q-2)-q\FR{k\!-\!2}2-\FR{(q\!-\!1)(q\!-\!3)}4+1.
\end{equation}
The right side is negative and decreases as $k$ increases.  The coefficient
on $k$ in the third factor on the left is $q/4-1$, which is positive, so the
factor increases as $k$ increases.  Hence it suffices to show that the
inequality cannot hold when $k$ takes its least allowed value, $q+1$.  The
inequality then simplifies to
$$
q\FR{q+1}2\left[\left(\FR q2-1\right)(q-4)-\FR{(q-1)q}4+\FR{3}2+\FR7q\right]
\le -(q-1)\left[\FR{(q+1)(q-2)}4+\FR{q}2+\FR{(q-3)}4\right]+1.
$$
The left side increases with $q$, and the right side decreases with $q$,
so it suffices to show that the inequality fails when $q=4$.
The left side is then $10/4$ and the right side is $-53/4$.

When $k$ is odd, we instead set $m=(k-3)/2$.  In this case $k-2m-2=1$
and $k-1-m=(k+1)/2$, so \eqref{prequad} becomes
\begin{align}\label{qkodd}
q\FR{k+1}2&\left[\left(\FR q2-1\right)(k-5)-\FR{(k-3)q}4+\FR32+\FR 7q\right]\\
&\le -\FR{(k-3)(k-1)(q-2)}4+(3q-7)-\FR{(k-3)q}2-\FR{(q-1)(q-3)}4+1.
\nonumber
\end{align}
Again the last factor on the left increases with $k$, so again
it suffices to consider the smallest allowed value for $k$, which is $q+1$.
We require
\begin{align}\label{kodd}
q\FR{q+2}2&\left[\left(\FR q2-1\right)(q-4)-\FR{(q-2)q}4+\FR32+\FR7q\right]\\
&\le -\FR{(q-2)q(q-2)}4+(3q-7)-\FR{(q-2)q}2-\FR{(q-1)(q-3)}4+1.
\nonumber
\end{align}
Again the left side increases with $q$ and the right side decreases with $q$,
so it suffices to show that the inequality cannot hold when $q=4$.
Then the left side is $15$ and the right is $-11/4$.

Thus the inequality cannot hold for any allowed values of the parameters,
and an extension must exist.
\hfill$\blacksquare$

\bigskip
Although the computation at the end of the proof works for $q=4$, other
difficulties arise when seeking this extension.  First, instead of $l\ge m+1$
we must also consider $l=m$.  Also, although most cases in the proofs of
Theorem~\ref{THMA} and Lemmas~\ref{ClaimD} and~\ref{ClaimE} extend to $q=4$
(sometimes with additional case analysis), the very special part of Case 2
in Lemma~\ref{ClaimD} when $(u,x,w,z)=(u_2,u_4,u_6,u_7)$ does not work.  This
can be fixed by changing $|S|$ from $3q-11$ to $3q-10$, but then the term $3/2$
on the left side of the numerical inequality beomes $1/2$, and the desired
contradiction fails to occur in the one special case $(q,k,l,m)=(4,5,1,1)$.
Further analysis for that case could complete the proof for $q=4$.

\section{Acknowledgments}
The authors would like to thank the	anonymous referees for their useful suggestions to simplify and improve the paper.

\end{document}